\documentclass[12pt,leqno]{amsart}

\usepackage{amsmath,stmaryrd}
\usepackage{amssymb}
\usepackage{amsthm}
\usepackage[english]{babel}
\usepackage{epsf}
\usepackage{graphicx}
\usepackage{esint} 
\usepackage{dsfont}
\usepackage[pagebackref]{hyperref} 
\hypersetup{pdfpagemode=FullScreen,  colorlinks=true}  
\usepackage{enumitem} 
\usepackage{soul,xcolor}

\setstcolor{red}

\theoremstyle{plain}
\newtheorem{Th}{Theorem}[section]
\newtheorem{Lemma}[Th]{Lemma}
\newtheorem{Cor}[Th]{Corollary}
\newtheorem{Prop}[Th]{Proposition}

 \theoremstyle{remark}
\newtheorem{Def}[Th]{Definition}

\newtheorem{Rem}[Th]{Remark}

\numberwithin{equation}{section}

\DeclareMathOperator{\Tr}{Tr}
\DeclareMathOperator{\dist}{dist}
\DeclareMathOperator{\diver}{div}
\DeclareMathOperator{\supp}{supp}
\DeclareMathOperator{\diam}{diam}
\DeclareMathOperator{\Id}{Id}

\newcommand{\R}{\mathbb R}
\newcommand{\E}{\mathcal E}

\newcommand{\A}{\mathcal A}
\newcommand{\wt}{\widetilde}
\newcommand{\bp}{\noindent {\em Proof. }}
\newcommand{\ep}{\hfill $\square$}
\newcommand{\1}{{\mathds 1}}

\title{Carleson Perturbations for the Regularity Problem}

\author[Dai]{Zanbing Dai}
\address{Zanbing Dai. School of Mathematics, University of Minnesota, Minneapolis, MN 55455, USA}
\email{dai00003@umn.edu}

\author[Feneuil]{Joseph Feneuil}
\address{Joseph Feneuil. Dipartimento di Matematica, Università di Pisa,
Largo Bruno Pontecorvo, 7, I–78349 Pisa, Italy}
\email{joseph.feneuil@dm.unipi.it}

\author[Mayboroda]{Svitlana Mayboroda}
\address{Svitlana Mayboroda. School of Mathematics, University of Minnesota, Minneapolis, MN 55455, USA}
\email{svitlana@math.umn.edu}

\thanks{S. Mayboroda was partly supported by the NSF
RAISE-TAQS grant DMS-1839077 and the Simons foundation grant 563916, SM. J. Feneuil was partially supported by the Simons foundation grant 601941, GD and by the European Research Council via the project ERC-2019-StG 853404 VAREG}

\begin{document}

\maketitle

\begin{abstract}
We prove that the solvability of the regularity problem in $L^q(\partial \Omega)$ is stable under Carleson perturbations. If the perturbation is small, then the solvability is preserved  in the same $L^q$, and if the perturbation is large, the regularity problem is solvable in $L^{r}$ for some other $r\in (1,\infty)$. We extend an earlier result from Kenig and Pipher to very general unbounded domains, possibly with lower dimensional boundaries as in the theory developed by Guy David and the last two authors.
To be precise, we only need the domain to have non-tangential access to its Ahlfors regular boundary, together with a notion of gradient on the boundary.
\end{abstract}

{\bf Keywords:} regularity problem, Dirichlet problem, degenerate elliptic equation, Carleson perturbations.

\medskip

{\bf MSC2020:} 35J25 (primary), 31B25, 35J70, 42B25.

\medskip

\tableofcontents

\section{Introduction}

\subsection{History and motivation}

In the last 40 years, and even more in the last 10 years, there have been impressive developments at the intersection of harmonic analysis, elliptic PDEs, and geometric measure theory.  Their main goal is to understand as much as possible the interaction between geometry of (the boundary of) a domain and bounds on solutions of boundary value problems.

The first important result beyond the complex plane is due to Dahlberg in \cite{dahlberg1977estimates} \cite{dahlbert1979poisson}, and it states that the Dirichlet problem is solvable in $L^2$ whenever the domain is Lipschitz. 
Since then, considerable efforts have been devoted to weakening the conditions on domains $\Omega$ and theirs boundaries, and to replacing the harmonic functions - that are solutions to $-\Delta u = 0$ - by solutions of elliptic operators in the form $L=-\diver A \nabla$. These two directions are not independent from each other, because with the help of  changes of variables, we can make $\partial \Omega$ smoother, and the price to pay is rougher coefficients for the matrix $A$.

As far as the Dirichlet boundary value problem is concerned, mathematicians in the area have a pretty clear picture. When the operator is the Laplacian, the solvability of the Dirichlet problem in $L^p$ for some large $p\in (1,\infty)$ is equivalent to the fact that the boundary of the domain $\partial \Omega$ is uniformly rectifiable of dimension $n-1$ (see \cite{david1991singular},\cite{david1993analysis} for the definition) and the domain has sufficient access to the boundary. A non-exhaustive list of works that helped to arrive to this conclusion includes \cite{david1990lipschitz}, \cite{semmes1989criterion}, \cite{hofmann2014uniform}, \cite{azzam2014new}, and \cite{azzam2020harmonic}. 
One cannot replace the Laplacian by a general uniformly elliptic operator and still preserve the $L^p$-solvability of the Dirichlet problem (see \cite{caffarelli1981completely}, \cite{modica1980construction}). The uniformly elliptic operators $L=-\diver A \nabla$ that preserve the $L^p$ solvability of the Dirichlet problem fall into two classes. The first one is the $t$-independent operators (see for instance \cite{jerison1981dirichlet}, \cite{kenig2000new}, and \cite{hofmann2015square}), and the second one is related to Carleson measures, either via perturbations (e.g. \cite{fabes1984necessary}, \cite{dahlberg1986absolute}, \cite{fefferman1989criterion}, and \cite{fefferman1991theory}) or via the oscillations of $A$ (a.k.a Dahlberg-Kenig-Pipher operators, see \cite{kenig2001dirichlet}, \cite{dindos2007lp}). Many of the results have been extended to complex valued elliptic operators and elliptic systems (\cite{hofmann2015regularity}, \cite{dindovs2019regularity}, \cite{dindovs2020boundary}, and \cite{dindovs2021lp}). For an interested reader, who is new to this area, a nice and detailled discussion on those topics can be found in the introduction of \cite{feneuil2020generalized}.

A natural question to ask is whether those results for the Dirichlet boundary value problem have analogues for other boundary value problems, such as the Neumann problem and the regularity problem. However, those problems appear to be considerably more complicated, 
some results are shown in \cite{jerison1981neumann}, \cite{verchota1984layer}, \cite{kenig1993neumann}, \cite{kenig1995neumann}, \cite{auscher2010weighted}, \cite{auscher2012weighted}, \cite{hofmann2015regularity}, and \cite{dindovs2017boundary}, but they do not go as far as one would expect, for instance they don't go beyond Lipschitz domains. 

In the recent impressive breakthrough \cite{mourgoglou2021regularity}, Mourgoglou and Tolsa have shown the solvability of the regularity problem in some Sobolev spaces for the Laplacian on open bounded domains satisfying the corkscrew condition and with uniformly rectifiable boundaries.  The key point is the use of an alternative to the classical boundary Sobolev space (called the Haj\l asz-Sobolev spaces) to bypass the lack of connectedness of the boundary of the domains. The importance of the Haj\l asz-Sobolev spaces is supported by a counterexample from the authors, that shows that the result is false when one uses the classical Sobolev spaces. Mourgoglou and Tolsa complete their article by giving additional geometric conditions (that we interpret as connectedness on the boundary - like the validity of a Poincar\'e inequality on boundary balls) for which the classical Sobolev spaces and the Haj\l asz-Sobolev spaces are the same, which ultimately give the existence of some non-Lipschitz domains where the regularity problem is solvable for the Laplacian in the classical Sobolev spaces. After the submission of our article, the two new manuscripts \cite{dindos2022regularity} and \cite{mourgoglou2022lp} successfully extended the solvability of the Regularity problem to all the Dahlberg-Kenig-Pipher operators, hence generalizing some results from \cite{dindovs2017boundary} and \cite{mourgoglou2021regularity}.

In our article, we look at the stability of the regularity problem under Carleson perturbations \cite{kenig1995neumann} on a ball, 
and we prove that we can extend it in several directions: first we consider operators which are not necessarily symmetric, second we extend the geometric setting to uniform domains - which are domains with non-tangential access and Ahlfors regular boundaries, using as Mourgoglou and Tolsa the Haj\l asz-Sobolev spaces - and third, we allow
low dimensional boundaries, which were studied for the Dirichlet problem by Guy David, Zihui Zhao, Bruno Poggi, and the two last authors (see \cite{david2017elliptic}, \cite{david2017dahlberg}, \cite{mayboroda2019square}, \cite{david2020elliptic}, \cite{mayboroda2020carleson}, \cite{feneuil2018dirichlet}, \cite{feneuil2020generalized}, \cite{david2020harmonic}, and \cite{feneuil2020absolute}). Combined with another paper under preparation (\cite{dai2021regularity}), we ultimately prove the solvability of the regularity problem on the complement of a Lipschitz graph of lower dimension.

\subsection{Introduction to the setting}

The aim of this subsection is to introduce results from \cite{david2017elliptic} and \cite{david2020elliptic} and to give basic definitions adapted to the setting at hand.

As mentioned in the previous subsection, we understand now that we can characterize the uniformly rectifiable sets $\Gamma \subset \R^n$ of dimension $n-1$ via some bounds on the oscillations of the bounded harmonic functions on $\Omega$ (or the solvability of the Dirichlet problem), 
where $\Gamma = \partial \Omega$ and $\Omega$ has enough access to its boundary. Guy David and the two last authors launched a program to
extend this characterization of the uniform rectifiability to uniformly rectifiable sets of lower dimension $d \leq n-2$. In this case, the domain $\Omega = \R^n \setminus \Gamma$ has plenty of access to its boundary (see Proposition \ref{lowuniform}). However, a bounded harmonic function in $\Omega$ is also a bounded harmonic function in $\R^n$, and thus does not ``see'' the boundary $\Gamma$. For that reason, the authors developed in \cite{david2017elliptic} an elliptic theory that is adapted to low-dimensional boundaries by using some operators whose coefficients are elliptic and bounded with respect to a weight. Let us give a quick presentation of this theory.

Consider a domain $\Omega \subset \R^n$ whose boundary is $d$-dimensional Ahlfors regular, that is, there exists a measure $\sigma$ supported on $\partial \Omega$  and $C_\sigma >$ such that  
\begin{align}\label{DEFSIG}
    C^{-1}_{\sigma}r^d\leq \sigma(\Delta(x,r))\leq C_\sigma r^d \quad \text{ for } x\in \partial \Omega, \, r>0,
\end{align}
where $\Delta(x,r):= B(x,r) \cap \partial \Omega$ is a boundary ball. If \eqref{DEFSIG} holds for some measure $\sigma$, then it works also with $\sigma':= \mathcal H^d_{|\partial \Omega}$, the $d$-dimensional Hausdorff on $\partial \Omega$. 
The incoming results would also be true for bounded domains when we ask \eqref{DEFSIG} only when $r\leq \diam(\Omega)$, but the proof would require splitting cases (even though the two cases are fairly similar) and we do not tackle it here. 

Observe that when $d<n-1$, we necessarily have that $\Omega = \R^n \setminus \partial \Omega$ and the domain $\Omega$ automatically has access to its boundary (see Proposition \ref{lowuniform}). When $d\geq n-1$, we assume that $\Omega$ satisfies the interior corkscrew point condition and the interior Harnack chain condition (see Definitions \ref{defCPC} and \ref{defHCC}), which means that $\Omega$ is 1-sided NTA and hence uniform.

\medskip

Consider a class of operators $\mathcal L=- \diver A \nabla$ on $\Omega$, where the coefficients are elliptic and bounded with respect to the weight $w(X):= \dist(X,\partial \Omega)^{d+1-n}$. To be more precise, we assume that there exists $\lambda>0$ such that
\begin{align}\label{ELLIP2}
\lambda |\xi|^2 w(X) \leq A(X)\xi\cdot \xi\ \ \text{and}\ \
|A(X)\xi\cdot \zeta|\leq \lambda^{-1} w(X) |\xi||\zeta|,\ \ \xi, \zeta \in \R^n\, X\in \Omega.
\end{align}
If we write $\mathcal A$ for the rescaled matrix $w^{-1} A$, then the operators that we consider are in the form $\mathcal L := -\diver [w\mathcal A \nabla]$ where $\mathcal A$ satisfies the classical elliptic condition
\begin{align}\label{ELLIP}
\lambda|\xi|^2\leq \mathcal{A}(X)\xi\cdot \xi\ \ \text{and}\ \
|\mathcal{A}(X)\xi\cdot \zeta|\leq \lambda^{-1} |\xi||\zeta|,\ \ \xi, \zeta \in \R^n\,  X\in \Omega.
\end{align}
A weak solution to $\mathcal L u = 0$ lies in $W^{1,2}_{loc}(\Omega)$ and satisfies 
\begin{align*}
    \int_\Omega \mathcal{A}\nabla u\cdot \nabla \varphi \, dm=0 \qquad \text{ for } \varphi \in C^\infty_0(\Omega),
\end{align*}
where $dm(Y) = w(Y) dY$.

The weak solutions to $\mathcal L u = 0$ satisfy De Giorgi-Nash-Moser estimates (interior and at the boundary). We can also construct a Green function for $\mathcal L$, an elliptic measure on $\partial \Omega$, and derive the comparison principle, a.k.a. CFMS estimates. The full elliptic theory is presented in Subsections \ref{SSelliptic} to \ref{SScomparison}.

There are two fairly standard ways to construct weak solutions. The first one is using the Lax-Milgram theorem in an appropriate weighted Sobolev space (see Lemma \ref{THLMG}). The second one, that will be the one used in the present article, is via the elliptic measure, which is a collection of probability measures $\{\omega^X\}_{X\in \Omega}$ such that, for every compactly supported continuous function $f$ on $\partial \Omega$, the function defined as
\begin{equation} \label{defhm2}
u_f(X) := \int_{\partial \Omega} f(x) \, d\omega^X(x)
\end{equation}
belongs to $C^0(\overline{\Omega})$, and is a weak solution to $\mathcal L u = 0$, and satisfies $u\equiv f$ on $\partial \Omega$. Note that \eqref{defhm2} will be used to provide a formal solution to 
\[\left\{\begin{array}{l} \mathcal L u = 0 \text{ in } \Omega \\
u = f \text{ on } \partial \Omega.
 \end{array}\right. \]

We are ready to introduce the Dirichlet boundary value problem.

\begin{Def}[Dirichlet problem]\label{def:drp}
The Dirichlet problem is solvable in $L^p$ if there exists $C>0$ such that for every $f\in C_c(\partial \Omega)$, the solution $u_f$ constructed by \eqref{defhm2} verifies
\begin{equation} \label{defDirichlet}
\|N(u_f)\|_{L^p(\partial \Omega,\sigma)} \leq C \|f\|_{L^p(\partial \Omega,\sigma)},
\end{equation}
where $N$ is the non-tangential maximal function defined as
\begin{equation} \label{DEFN}
N(v)(x):= \sup_{\gamma(x)} |v|
\end{equation}
and $\gamma(x) := \{X\in \Omega,\, |X-x| \leq 2\delta(X)$\} is a cone with vertex at $x\in \partial \Omega$. 
\end{Def}

In Definition \ref{def:drp}, the data $f$ lies in $C_c(\partial \Omega)$ instead of $L^p(\partial \Omega,\sigma)$,  so that we have a way to construct $u_f$ {\em a priori} using the harmonic measure. Once we know that \eqref{defDirichlet} holds for any $f\in C_c(\partial \Omega)$, we can construct {\em a posteriori} the solutions $u_f$ for any $f\in L^p(\partial \Omega,\sigma)$ by density, and those solutions will satisfy \eqref{defhm2} and \eqref{defDirichlet}.

\subsection{Main results}
We recall that we use $\delta(X)$ for $\dist(X,\partial \Omega)$, $w(X)$ for $\delta(X)^{d+1-n}$, $dm(X)$ for $\delta(X)^{d+1-n} dX$, and $B_X$ for $B(X,\delta(X)/4)$. In this section, we consider two elliptic operators $\mathcal{L}_0,\mathcal{L}_1$ in the form $\mathcal L_i = -\diver [w\mathcal A_i \nabla]$, where $\A_i$ is real, not necessarily symmetric, and uniformly elliptic \eqref{ELLIP}.

We define the disagreement between $\A_0$ and $\A_1$ as:
\begin{align}\label{DEFDF}
    a(X):=\sup_{Y\in B_{X}} |\mathcal{E}(Y)|,\ \ \mathcal{E}(Y):=\A_0(Y)-\A_1(Y).
\end{align}
Assume that $\delta(X)^{d-n}|a(X)|^2dX$ is a Carleson measure, that is
 \begin{align} \label{DEFCA}
    \int_{B(x,r) \cap \Omega} |a(X)|^2\frac{dX}{\delta(X)^{n-d}}\leq Mr^d.
\end{align}

The stability of the solvability of the Dirichlet problem under Carleson perturbations was established in \cite{fefferman1991theory} (when $\Omega$ is a ball), \cite{cavero2019perturbations} and \cite{cavero2020perturbations} 
(when $\Omega$ is a uniform domain in with $n-1$-dimensional boundary),  \cite{mayboroda2020carleson} (when $\Omega$ is uniform with lower dimensional boundary), and \cite{feneuil2020generalized} 
(1-sided NTA domains and enough basic bounds on the harmonic measure, a setting that includes all the previous ones and more).
These results are as follows: 

\begin{Th}[\cite{fefferman1991theory,cavero2020perturbations,mayboroda2020carleson,feneuil2020generalized}] \label{ThDirichlet}
Let $\Omega$ be a uniform domain and let $\mathcal L_0,\mathcal L_1$ be two elliptic operators whose coefficients are real, non necessarily symmetric, and uniformly elliptic in the sense \eqref{ELLIP}. Assume that the Dirichlet problem for the operator $\mathcal L_0$ is solvable in $L^{p_0}$ (see Definition \ref{def:drp}).
If the disagreement \eqref{DEFDF} satisfies the Carleson measure condition \eqref{DEFCA}, then there exists $p_1\in (1,\infty)$ such that the Dirichlet problem for $\mathcal L_1$ is solvable in $L^{p_1}$.

Moreover, there exists $\epsilon_0>0$ (that depends on $p_0$ and $\mathcal L_0$) such that if the Carleson norm $M$ in \eqref{DEFCA} is smaller then $\epsilon_0$, then the Dirichlet problem for $\mathcal L_1$ is solvable in the same $L^{p_0}$.
\end{Th}

We know that if the Dirichlet problem is solvable in $L^p$, then it is solvable in $L^{r}$ for all $r\in (p,\infty)$. In this sense, $p_1$ is {\em a priori} larger than $p_0$. The second part of the theorem says that if the disagreement is small enough, then we can preserve the solvability of the Dirichlet problem in the same $L^{p_0}$ space. 

Our main theorem addresses this type of perturbations for the regularity problem. The only forerunner in this case is \cite{kenig1995neumann}. However, contrary to \cite{kenig1995neumann}, here we treat the domains with rough and/or low dimensional boundaries, and operators whose coefficients are not necessarily symmetric. Let us state exactly what we prove.

\begin{Th}\label{thm:mn1}
Let $\Omega$ be a uniform domain (see Definition \ref{defuniform}) and let $\mathcal L_0,\mathcal L_1$ be two elliptic operators whose coefficients are real, non necessarily symmetric, and uniformly elliptic in the sense \eqref{ELLIP}. Assume that the Dirichlet problem for the adjoint operator $\mathcal L_1^*$ is solvable in $L^{q'}$ (see Definition \ref{def:drp}).

If the disagreement \eqref{DEFDF} satisfies the Carleson measure condition \eqref{DEFCA}, then for any $f\in C_c(\partial \Omega) $, the two solutions $u_{0,f}$ and $u_{1,f}$ to $\mathcal L_0 u_{0,f} = 0$ and $\mathcal L_1 u_{1,f} = 0$ constructed by \eqref{defhm2} verify
\begin{equation} \label{main1a}
\|\wt N(\nabla u_{1,f})\|_{L^q(\partial \Omega,\sigma)} \leq C M \|\wt N(\nabla u_{0,f})\|_{L^q(\partial \Omega,\sigma)},
\end{equation}
where $\wt N$ is the averaged non-tangential maximal function defined as
\begin{equation} \label{defwtN}
\wt N(v)(x):= \sup_{X\in \gamma(x)} \left(\fint_{B_X} |v|^2\, dX\right)^\frac12.
\end{equation}
The two quantities in \eqref{main1a} can be infinite, but the left-hand side has to be finite as soon as the right-hand side is finite. The constant $C>0$ depends only on $n$, the uniform constants of $\Omega$, the ellipticity constant $\lambda$, the parameter $q$, and the constant in \eqref{defDirichlet}. 
\end{Th}

The above theorem assumes the solvability of the Dirichlet problem for $\mathcal L_1^*$. It is not very surprising, because it can be seen as a partial converse of Theorem \ref{ThRq=>Dq'} below. That is, if the Dirichlet problem for $\mathcal L^*_1$ is solvable in $L^{q'}$, and if $\mathcal L_1$ satisfies extra conditions, then the regularity problem for $\mathcal L_1$ is solvable in $L^q$.

We wanted to state the above theorem independently of the notion of regularity problem. We remark that it can also be used for the Neumann problem, although not directly, and we leave this question for a future article.

\medskip

We turn now to the definition of the regularity problem, which is long overdue.

\subsection{The regularity problem}

Informally speaking, the regularity problem in $L^q$ comes down to a bound on $\|\wt N(\nabla u_f)\|_{L^q(\partial \Omega)}$ in terms of the tangential derivatives of $u_f$ on the boundary (i.e. in terms of the derivatives of $f$). If $\partial \Omega = \R^d$ is a plane, then we want to show that 
\begin{equation} \label{defRqA} 
\|\wt N(\nabla u_f)\|_{L^q(\R^d)} \leq C \|\nabla_{\R^d} f\|_{L^q(\R^d)},
\end{equation}
where here $\nabla_{\R^d}$ is the classical gradient in $\R^d$. Similarly, on a more complicated boundary, the regularity problem would correspond to the estimate
\begin{equation} \label{defRqB} 
\|\wt N(\nabla u_f)\|_{L^q(\partial \Omega,\sigma)} \leq C \|\nabla_{\partial \Omega,q} f\|_{L^q(\partial \Omega,\sigma)},
\end{equation}
where $\nabla_{\partial\Omega,q}$ is a notion of (tangential) gradient that may depend on $\partial \Omega$ and $q$. If $\partial \Omega = \Gamma$ is the graph of a Lipschitz function, then we can extend the notion of gradient from $\R^d$ to $\Gamma$, for instance by finding a bi-Lipschitz map $\rho : \, \Gamma \to \R^d$ and define
\[\nabla_\Gamma f(x) := \nabla_{\R^d} [f\circ \rho] \circ \rho^{-1}(x) \qquad \text{ for almost every } x\in \partial \Omega.\]
On the other hand, let us take $\partial \Omega := P_1 \cup P_2$ to be the union of two low dimensional planes that do not intersect (and then $\Omega = \R^n \setminus \partial \Omega)$. We have a well-defined gradient $\nabla_{P_1 \cup P_2}$ on $\partial \Omega$ (because we have a gradient on planes), we also have elliptic operators and solutions thanks to the elliptic theory from \cite{david2017elliptic}. However, since $P_1 \cup P_2$ is not connected, $\nabla_{P_1 \cup P_2} f = 0$ does not necessarily imply that $f$ is constant on $P_1 \cup P_2$, and thus we can never have
\begin{equation} \label{defRqC} 
\|\wt N(\nabla u_f)\|_{L^q(P_1 \cup P_2)} \leq C \|\nabla_{P_1 \cup P_2} f\|_{L^q(P_1 \cup P_2)}.
\end{equation}
Recall that Theorem \ref{thm:mn1} only requires $\Omega$ to be uniform, and so nothing can stop $\Omega$ from being very rough (even purely unrectifiable) and not connected. So if we do not want to lose too much from Theorem \ref{thm:mn1}, we would prefer a notion of gradient that exists for any set, and for any function obtained by restricting the ones from $C^\infty_0(\R^n)$ to $\partial \Omega$.

Fortunately for us, Lipschitz functions exist on every metric space, and a notion of gradient was derived from it. For a Borel function $f: \, \partial \Omega \to \R$, we say that a non-negative Borel function $g$ is {\em a generalized gradient of $f$} if 
\begin{equation}
|f(x) - f(y)| \leq |x-y| (g(x) + g(y)) \qquad \text{ for } \sigma\text{-a.e. } x,y\in \partial \Omega.
\end{equation}
The collection of all generalized gradients is denoted by $D(f)$. Then for any $p\geq 1$, the space $\dot W^{1,p}(\partial \Omega,\sigma)$ is the space of Borel functions that have a generalized gradient in $L^p$, and we equip it with the semi-norm
\begin{equation}
\|f\|_{\dot W^{1,p}(\partial \Omega,\sigma)} := \inf_{g\in D(f)} \|g\|_{L^p(\partial \Omega,\sigma)}.
\end{equation}
Haj\l asz introduced those spaces in \cite{hajlasz1996sobolev}, that is why they are often called Haj\l asz-Sobolev spaces. We refer an interested reader to \cite{heinonen2005lectures} for more information on Sobolev spaces on general metric spaces.

It should not a big surprise to bring up Haj\l asz-Sobolev spaces here, since they have already been used to study boundary values problems in \cite{hofmann2010singular} and recently in \cite{mourgoglou2021regularity}. In the latter article, the authors proved that in bounded domains with $(n-1)$-dimensional uniformly rectifiable boundaries, the solvability in $L^p$ of the Dirichlet problem for the Laplacian is equivalent to the solvability of the regularity problem, defined below with the Haj\l asz-Sobolev space.

Note that any function $f$ which lies in a Haj\l asz-Sobolev space supports a Poincar\'e inequality, that is for any $p\in [1,\infty]$ and for any boundary ball $\Delta = \Delta(x,r)$, we have
\begin{equation} \label{PoincareHS}
\|f - f_\Delta\|_{L^p(\Delta,\sigma)} \leq C_p r \inf_{g\in D(f)} \|g\|_{L^p(\Delta,\sigma)}
\end{equation}
where $f_\Delta = \fint_\Delta f\, d\sigma$. The proof of this fact is immediate. Indeed, if $g\in D(f)$, we have
\begin{multline*}
\fint_\Delta |f-f_\Delta|^p \, d\sigma \leq \fint_\Delta \fint_\Delta |f(x)-f(y)|^p d\sigma(x) \, d\sigma(y) \leq \fint_\Delta \fint_\Delta |x-y|^p (g(x)+g(y))^p d\sigma(x) \, d\sigma(y) \\
\leq C_p \fint_\Delta g(x)^p \fint_\Delta |x-y|^p d\sigma(y) \, d\sigma(x) \leq C_p r^p \fint_\Delta g(x)^p d\sigma(x),
\end{multline*}
where we used the symmetry of the roles of $x$ and $y$ between the first and the second line.

\begin{Def}[Regularity problem]\label{def:rep}
The regularity problem is solvable in $L^p$ if there exists $C>0$ such that for every compactly supported Lipschitz function $f$, the solution $u_f$ constructed by \eqref{defhm2} verifies
\begin{equation}
\|\wt N(\nabla u_f)\|_{L^p(\partial \Omega,\sigma)} \leq C \|f\|_{\dot W^{1,p}(\partial \Omega,\sigma)},
\end{equation}
where $\dot W^{1,p}(\partial \Omega,\sigma)$ is the Haj\l asz Sobolev space defined above, and $\wt N$ is defined in \eqref{defwtN}.
\end{Def}

The regularity problem targets the question whether the oscillations of $u$ can be controlled by the oscillations of its trace, in a similar way that in the Dirichlet problem, $u$ is controlled by its trace. We replace $N$ by $\wt N$ because, contrary to $u_f$ which lies in $L^\infty_{loc}(\Omega)$ thanks to the Moser estimate \eqref{icaeq222}, we can only be certain of the fact that $\nabla u_f$ lies in $L^2_{loc}(\Omega)$.

It would be reassuring to know that the Haj\l asz-Sobolev spaces are the classical Sobolev spaces in the basic settings, which is not obvious at the first glance. We have indeed:
\begin{multline} \label{HS=S}
\text{when $\partial \Omega$ is a plane or the graph of a Lipschitz function,} \\ \text{$\dot W^{1,p}(\partial \Omega,\sigma)$ is the classical homogeneous Sobolev space.}
\end{multline}
The proof of the equivalence is a consequence of Lemma 6.5 in 
\cite{mourgoglou2021regularity}\footnote{The equivalence is established in a much more general situation, involving uniformly rectifiable sets and Poincar\'e inequalities.}.

Note also that the Haj\l asz-Sobolev spaces contain the compactly supported Lipschitz functions, so they cannot be too small. Besides, we have the following duality result between regularity and Dirichlet problems proven in the appendix.

\begin{Th} \label{ThRq=>Dq'}
Let $\Omega$ be a uniform domain, and let $\mathcal L = - \diver [w\A \nabla]$ be an elliptic operator whose coefficients satisfy \eqref{ELLIP}. 

If the regularity problem (defined with the Haj\l asz-Sobolev spaces) for $\mathcal L$ is solvable in $L^q$ for some $q\in (1,\infty)$, then the Dirichlet problem for the adjoint $\mathcal L^*$ is solvable in $L^{q'}$, where $\frac1q + \frac1{q'} = 1$.
\end{Th}

The combination of Theorems \ref{ThDirichlet}, \ref{thm:mn1}, and \ref{ThRq=>Dq'} gives the following corollary.

\begin{Cor}\label{thm:mn5}
Let $\Omega$ be a uniform domain, and let $\mathcal L_0,\mathcal L_1$ be two elliptic operators whose coefficients are uniformly elliptic in the sense of \eqref{ELLIP}. Assume that  the disagreement \eqref{DEFDF} satisfies the Carleson measure condition \eqref{DEFCA}. Then the following holds.
\begin{enumerate}
\item If there exists $q_0\in (1,\infty)$ such that the regularity problem for the operator $\mathcal L_0$ is solvable in $L^{q}$ for any $q\in (1,q_0]$, then there exists $q_1\in (1,\infty)$ such that the regularity problem for $\mathcal L_1$ is solvable in $L^q$ for any $q\in (1,q_1]$.
\item  If there exists $q_0\in (1,\infty)$ such that the regularity problem for the operator $\mathcal L_0$ is solvable in $L^{q_0}$, and if the Carleson norm $M$ in \eqref{DEFCA} is smaller than  $\epsilon_0$ (depending only on $q_0$ and $\mathcal L_0$), then the regularity problem for $\mathcal L_1$ is solvable in the same $L^{q_0}$.
\end{enumerate}
\end{Cor}

\begin{Rem}
When the boundary is smooth (flat or Lipschitz), the Haj\l asz-Sobolev spaces and the regular Sobolev spaces are the same \eqref{HS=S}, and the solvability of the regularity problem for an operator $\mathcal L$ in $L^{q}$ immediately implies the solvability of the regularity problem in $L^{p}$, $p\in (1,q]$. The proof of this result is the same as the one of Theorem 5.2 in \cite{kenig1993neumann} (see also \cite{dindos2011regularity}), which treats the case of bounded domains with smooth boundary. 

However, the proof cannot be directly adapted to our context, because the proof relies on the properties of Hardy-Sobolev spaces on the boundary, which are not constructed yet in our setting that uses the generalized gradient.
\end{Rem}

{\noindent \em Proof of the corollary.} 
Let $f\in C_c(\R^n)$ and let $u_{0,f}$ and $u_{1,f}$ be the two solutions to $\mathcal L_0 u_{0,f} = 0$ and $\mathcal L_1 u_{1,f} = 0$ constructed by \eqref{defhm2}. Theorem \ref{ThRq=>Dq'} shows that the Dirichlet problem for $\mathcal L_0^*$ is solvable in $L^{q_0'}$, and then Theorem \ref{ThDirichlet} implies that the Dirichlet problem for $\mathcal L_1^*$ is solvable in $L^{q'_1}$ for some $q'_1\in [q'_0,\infty)$. Theorem \ref{thm:mn1} further provides the estimate
\begin{equation} \label{corzz1}
\|\wt N(\nabla u_{1,f})\|_{L^{q_1}(\partial \Omega,\sigma)} \leq C \|\wt N(\nabla u_{0,f})\|_{L^{q_1}(\partial \Omega,\sigma)},
\end{equation}
where $\frac1{q_1} + \frac1{q'_1} = 1$.

Since $q'_1 \geq q'_0$, we have $q_0 \leq q_1$ and hence the regularity problem for $\mathcal L_0$ is solvable in $L^{q_1}$, that is
\begin{equation} \label{corzz2}
\|\wt N(\nabla u_{0,f})\|_{L^{q_1}(\partial \Omega,\sigma)} \leq C \|\nabla f\|_{L^{q_1}(\partial \Omega,\sigma)}
\end{equation}
We combine \eqref{corzz1} and \eqref{corzz2} to get that $\|\wt N(\nabla u_{1,f})\|_{q_1} \leq C \|\nabla f\|_{q_1}$, which concludes the first part of the corollary.

When $M$ is small, Theorem \ref{ThDirichlet} says that we can take $q'_1 = q'_0$ - hence $q_1 = q_0$ - in the above reasoning. The second part of the corollary follows.
\ep

\subsection{Conditions on the operator implying the solvability in $L^p$ of the regularity problem}

\subsubsection{Combination with \cite{dai2021regularity}}

In another article \cite{dai2021regularity}, we use the perturbation result from the present article to show that the regularity problem is solvable in $L^2$ for a certain class of elliptic operators on $\R^n \setminus \R^d$, $d<n-1$. 

\begin{Th}[Theorem 1.1 in \cite{dai2021regularity}] \label{ThReg}
Let $d<n-1$ and $\Omega = \R^n \setminus \R^d := \{(x,t) \in \R^d \times \R^{n-d}, \, t \neq 0\}$. Assume that the operator $\mathcal L = -\diver [ |t|^{d+1-n} \mathcal A \nabla]$ satisfies \eqref{ELLIP} and is such that the matrix $\mathcal A$ can be written
\begin{equation} \label{A=B}
\mathcal A = \mathcal B + \mathcal C, \quad \text{ with } \quad   \mathcal B = \begin{pmatrix} \mathcal B_{||} & 0 \\ 0 & b \Id_{n-d} \end{pmatrix},
\end{equation}
and 
 \begin{align} \label{defca2}
    \int_{B(x,r)} \int_{|t|<r} \Big(|t||\nabla \mathcal B| + |\mathcal C|\Big)^2 \, \frac{dt\, dx}{|t|^{n-d}} \leq M r^d \qquad \text{ for } x\in \R^d, \, r>0.
\end{align}

There exists $\epsilon_0>0$ depending only on $n$, $d$, and the ellipticity constant $\lambda$ such that if the constant $M$ in \eqref{defca2} is smaller than $\epsilon_0$, then the regularity problem for $\mathcal L$ is solvable in $L^2$. 
\end{Th}

Thanks to \eqref{HS=S}, the solvability of the regularity problem above means that for any compactly supported Lipschitz function $f$ on $\partial \Omega$, the solution $u_f$ constructed by \eqref{defhm2} verifies 
\[\|\wt N(\nabla u_f)\|_{L^2(\R^d)} \leq C \|\nabla_t f\|_{L^2(\R^d)},\]
where the constant $C>0$ depends only on $d$, $n$, and $\lambda$, and where  $\nabla_t$ is the (tangential) gradient on $\R^d$.

If the matrix $\mathcal C$ is not included, Theorem \ref{ThReg} can be seen as the higher co-dimensional analogue of Theorem 5.1 in \cite{dindovs2017boundary}. The perturbation theory that we developed here allows us to add such term $\mathcal C$ to the coefficients of the operator.

\medskip

The objective of the project that includes both the present article and \cite{dai2021regularity} is to prove the solvability of the regularity problem when $\Omega$ and $\mathcal L$ are like those in \cite{david2017dahlberg}, that is when $\Omega$ is the complement of a Lipschitz graph of low dimension.

In domains with codimension 1 boundary, such results are classically obtained by using a change of variables that turns the Lipschitz domain into $\R^{n}_+$ (if the domain is unbounded) or a ball (if the domain is bounded). Such gain in regularity on the boundary is paid for by less regularity on the coefficients of the elliptic operators $\mathcal L$. In the case of Lipschitz domains with codimension 1 boundary, the change of variable used to flatten Lipschitz boundaries turns smooth operators like the Laplacian to operators in the form $-\diver \mathcal B \nabla$ with $\mathcal B$ like in \eqref{defca2}, see \cite{kenig2001dirichlet}. That is, the perturbation theory is not needed in this case.

However, the change of variable from \cite{kenig2001dirichlet} is not suitable to flatten Lipschitz graphs of low dimension, and another change of variable is needed, like the one in \cite{david2017dahlberg}. This second change of variable is almost isometric in the non-tangential direction, and the conjugate elliptic operator will have coefficients in the form \eqref{A=B} and \eqref{defca2}. Thus we can deduce from Theorem \ref{ThReg} the solvability of the regularity problem on the complement of a small Lipschitz graph of low dimension.

\begin{Cor} 
Let $d<n-1$. Let $\Gamma$ be the graph of a Lipschitz function $\varphi$. Consider the domain $\Omega := \R^n \setminus \Gamma$ and  the operator $L_\alpha = -\diver D_\alpha^{d+1-n} \nabla$, where $D_\alpha$ is the regularized distance
\[D_\alpha(X) := \left( \int_\Gamma |X-y|^{-d-\alpha} \, d\mathcal H^d(y) \right)^{-\frac1\alpha},\]
$\mathcal H^d$ is the $d$-dimensional Hausdorff measure, and $\alpha >0$.
 
There exists $\epsilon_0>0$ that depends only on $\alpha$ and $n$ such that if the Lipschitz constant $\|\nabla \varphi\|_\infty$ is smaller than $\epsilon_0$, then the regularity problem for $\mathcal L$ is solvable in $L^2$, which means that, for any compactly supported Lipschitz function $f$ on $\partial \Omega$, the solution $u_f$ constructed by \eqref{defhm2} verifies 
\[\|\wt N(\nabla u_f)\|_{L^2(\Gamma)} \leq C \|\nabla_t f\|_{L^2(\Gamma)},\]
where the constant $C>0$ depends only on $\alpha$ and $n$, and where  $\nabla_t$ is the (tangential) gradient on $\Gamma$.
\end{Cor}

\noindent {\em Proof of the corollary.} We use the bi-Lipschitz change of variable $\rho$ constructed in \cite{david2017dahlberg}. So the sovability of the regularity problem for $L_\alpha$ (defined on $\R^n \setminus \Gamma$) in $L^q$ is equivalent to the solvability regularity problem in $L^q$ for an operator $\mathcal L_\rho = - \diver [|t|^{d+1-n} \mathcal A_\rho \nabla]$ (defined on $\R^n \setminus \R^d$) where $\A_\rho$ satisfies
\[\mathcal A_\rho = \begin{pmatrix} \mathcal B_{||} & 0 \\ 0 & b \Id_{n-d} \end{pmatrix} + \mathcal C = \mathcal B + \mathcal C\]
and
 \begin{align} \label{defca3}
    \int_{B(x,r)} \int_{|t|<r} \Big(|t||\nabla \mathcal B|+ |\mathcal C| \Big)^2 \, \frac{dt\, dx}{|t|^{n-d}} \leq C_\alpha \|\nabla \varphi\|_\infty  r^d \qquad \text{ for } x\in \R^d, \, r>0.
\end{align}
The corollary follows now from Theorem \ref{ThReg}.
\ep

\subsubsection{Combination with \cite{mourgoglou2021regularity}} 
In uniform domains, Corollary 1.6 and part (a) of Corollary 1.7 in \cite{mourgoglou2021regularity}\footnote{The results in \cite{mourgoglou2021regularity} are stronger because they do not require the existence of Harnack chains inside the domain, like we do.}
gives:

\begin{Th}[\cite{mourgoglou2021regularity}]
If $\Omega$ be a bounded uniform domain with a uniformly rectifiable boundary, then the regularity problem (defined with the Haj\l asz-Sobolev spaces) for the Laplacian is solvable in $L^q$ for any $q\in (1,q_0]$ where $q_0>1$ sufficiently small.
\end{Th}

Thanks to Corollary \ref{thm:mn5}, we know that the above result can be extended to perturbations of the Laplacian as well.

\subsection{Plan of the article}\label{bprff}

A brief summary of this article is as follows. In Section \ref{Selliptic}, we state the precise statement of the assumptions on our domain, and we recall the elliptic theory that shall be needed for the proof of Theorem \ref{thm:mn1}. 
In particular, we construct the elliptic measure and we link it to the solvability of the Dirichlet problem. In Section \ref{Sinfty}, we construct the elliptic measure and Green function with pole at infinity, which are very convenient tools to deal with unbounded domains. 
Section \ref{Sproof} is devoted to the proof of Theorem \ref{thm:mn1}.

\medskip

In the rest of the article, we shall use $A \lesssim B$ when there exists a constant $C$ such that $A \leq C\, B$, where the dependence of $C>0$ into the parameters will be either obvious from context or recalled. We shall also write $A\approx B$ when $A\lesssim B$ and $B \lesssim A$.

\section{Our assumptions and the elliptic theory}

\label{Selliptic}

\subsection{Our assumptions on the domain}

In addition to the fact that the boundary $\partial \Omega$ is $d$-dimensional Ahlfors regular - see \eqref{DEFSIG} - we assume two extra hypotheses on the domain: the interior corkscrew point condition (quantitative openness)  and the interior Harnack chain condition (quantitative connectedness).
\begin{Def}[Corkscrew point condition] \label{defCPC}
We say that $\Omega$ satisfies the corkscrew point condition if there exists $c_{0}>0$ such that, for any $x\in \partial \Omega$ and any $r>0$, there exists $X\in B(x,r) \cap \Omega$ such that $B(X,c_0r) \subset \Omega$. 

A such point $X$ is called the corkscrew point associated to $x$ and $r$. Sometimes, the pair $(x,r)$ will be given by a boundary ball $\Delta$, that is we say that $X$ is a corkscrew point associated to a boundary ball $\Delta$ if $X$ is a corkscrew point associated to $x$ and $r$ where $\Delta = \Delta(x,r):=B(x,r) \cap \partial \Omega$.
\end{Def}

\begin{Def}[Harnack chain condition] \label{defHCC}
We say that $\Omega$ satisfies the Harnack chain condition if the following holds. For any $K>1$, there exists an integer $N_K$ such that for any $X,Y\in \Omega$ that satisfies $|X-Y| \leq K  \min\{\delta(X),\delta(Y)\}$, there exist at most $N_K$ balls $B_1,\dots B_{N_K}$ such that
\begin{enumerate}
\item $X\in B_1$ and $Y\in B_{N_K}$,
\item $2B_i \in \Omega$ for $1 \leq i \leq N_K$,
\item $B_i \cap B_{i+1} \neq \emptyset$ for $1\leq i \leq N_K - 1$.
\end{enumerate}
\end{Def}

\begin{Def}[Uniform domain] \label{defuniform}
We say that $\Omega$ is uniform if $\Omega$ satisfies the corkscrew point condition and the Harnack chain condition, and if $\partial \Omega$ is Ahlfors regular.

The constants in \eqref{DEFSIG}, Definition \ref{defCPC}, and Definition \ref{defHCC} are refered as ``the uniform constants of $\Omega$''. 
\end{Def}

Lemmas 11.7 and 2.1 of \cite{david2017elliptic} show that:

\begin{Prop} \label{lowuniform}
Let $d<n-1$, and let $\Omega = \R^n \setminus \Gamma$. If $\Gamma = \partial \Omega$ is $d$-dimensional Ahlfors regular, then $\Omega$ is uniform.
\end{Prop}

\subsection{Quantitative version of absolutely continuity}

In this article, we focus on doubling measures on $\partial \Omega$, which are non-negative Borel measures $\mu$ that verify
\begin{equation}
\mu(2\Delta) \leq C_\mu \mu(\Delta) \quad \text{ for any boundary ball } \Delta \subset \partial \Omega. 
\end{equation}
Two measures that will be considered in this paper are the Ahlfors regular (hence doubling) measure $\sigma$  and the elliptic measure with pole at infinity $\omega$ that will be constructed in Section \ref{Sinfty} and is doubling according to Lemma \ref{DBPWI}. Those two measures will be  comparable, more precisely $A^\infty$-absolutely continuous with each other, and the definition is given below.

\begin{Def}[$A_\infty$-absolute continuity]  \label{defAinfty}
Let $\nu,\mu$ be two doubling measures on $\partial \Omega$. We say that $\mu$ is  $A_\infty$-absolute continuous with respect to $\nu$ - or $ \mu \in A_\infty(\nu)$ in short - if for each $\epsilon>0$, there exists $\xi=\xi(\epsilon)>0$ such that for every surface ball $\Delta$, and every Borel set $E\subset \Delta$, we have that:
\begin{align}\label{AINF.eq00}
    \frac{\nu(E)}{\nu(\Delta)}<\xi\implies \frac{\mu(E)}{\mu(\Delta)}<\epsilon.
\end{align}
\end{Def}

The $A^\infty$-absolute continuity is related to the following stronger property.

\begin{Def}[The Reverse H\"older class $RH_p$] \label{defRHp} Let $\nu,\mu$ be two doubling measures on $\partial \Omega$ that are absolutely continuous with respect to each other. We say that $\nu \in RH_p(\mu)$ if there exists a constant $C_p$ such that for every surface ball $\Delta$, the Radon-Nikodym derivative $k=\frac{d\nu}{d\mu}$ satisfies
\begin{align}\label{RHQ.eq00}
    \Big ( \fint_{\Delta}|k|^p d\mu\Big )^{1/p}\leq C_p \fint_{\Delta}k d\mu = C_p \frac{\nu(\Delta)}{\mu(\Delta)}.
\end{align}
\end{Def}

The $A_\infty$ and $RH_p$ class satisfy several important properties, which are recalled here. 

\begin{Th}[Properties of $A_\infty$ measures; Theorem 1.4.13 of \cite{kenig1994harmonic};\cite{stromberg1989weighted}]\label{THCABO}
Let $\mu,\nu$ be two doubling measures on $\partial \Omega$, and let $\Delta$ be a surface ball. The following statements hold.
\begin{enumerate}[label=(\roman*)]
    \item\label{CB.01} If $\mu\in A_\infty(\nu)$, then $\nu$ is absolutely continuous with respect to $\nu$ on $\Delta$.
    \item\label{CB.02} The class $A_\infty$ is an equivalence relationship, that is $\mu\in A_\infty(\nu)$ implies $\nu \in A_\infty(\mu)$. 
    \item\label{CB.03} We have that $\mu \in A_\infty(\nu)$ if and only if there exist a constant $C>0$ and $\theta>0$ such that for each surface ball $\Delta$ and each Borel set $E\subset \Delta$, we have that
    \begin{align*}
        \frac{\mu(E)}{\mu(\Delta)}\leq C\Big (\frac{\nu(E)}{\nu(\Delta)}\Big )^{\theta}.
    \end{align*}
    \item\label{CB.05} $\mu \in A_\infty(\nu)$ if and only if we can find $p>1$ such that $\mu \in RH_p(\nu)$, i.e.
    \[A_\infty(\nu)=\bigcup\limits_{p>1}RH_p(\nu).\]
    \item\label{CB.06} $\mu \in RH_p(\nu)$ if and only if the uncentered Hardy-Littlewood maximal function with the measure $\mu$ defined as
    \begin{align*}
        (\mathcal M_\mu f)(x):=\sup_{\Delta \ni x} \fint_{\Delta'} |f| d\mu
    \end{align*}
    verifies the estimate
    \begin{align*}
        \|\mathcal M_\mu f\|_{L^{p'}(\partial \Omega, \nu)}\leq C \|f\|_{L^{p'}(\partial \Omega ,\nu)} \qquad \text{ for } f\in L^{p'}(\partial \Omega,\nu),
    \end{align*}
    where $p'$ is the H\"older conjugate of $p$, that is $\frac{1}{p}+\frac{1}{p'}=1$.
\end{enumerate}
\end{Th}

\subsection{The basic elliptic theory} \label{SSelliptic}

To lighten the notations, in the rest of the article, we shall write $\delta(X)$ for $\dist(X,\partial \Omega$), $w(X)$ for $\delta(X)^{d+1-n}$, $dm(X)$ for $w(X)dX$, and $B_X$ for $B(X,\delta(X)/4)$. The measure $m$ is doubling, as shown in Lemma 2.3 of \cite{david2017elliptic}, but more importantly $m$ satisfies some boundary and interior Poincar\'e inequalities (Lemma 4.2 in \cite{david2017elliptic} when $d<n-1$, and Theorem 7.1 in \cite{david2020elliptic} for the statement in any dimension).

\medskip

The correct function spaces to study our elliptic equations are the weighted Sobolev space
\begin{align}\label{def:wsp}
W:=\{u\in L^1_{loc}(\Omega): \, \|u\|_W:= \|\nabla u\|_{L^2(\Omega,dm)} < +\infty\}
\end{align}
and the space of traces 
\begin{align}\label{def:hspa}
H:=\{f: \, \|f\|_H:= \int_{\partial\Omega} \int_{\partial\Omega} \frac{|f(x)-f(y)|}{|x-y|^{d+1}}d\sigma(x)d\sigma(y)<\infty \}.
\end{align}
For those spaces, we can construct a bounded trace operator $\Tr : \, W \to H$, by trace operator, we mean that $\Tr (u) = u$ whenever $u\in W \cap C^0(\overline{\Omega})$, which is uniquely defined by the density of $W \cap C^0(\overline{\Omega})$ in $W$ (see Lemma 9.19 in \cite{david2020elliptic}). 
We shall also need
\[W_0 := \{u\in W, \, \Tr(u) = 0\},\]
which is also the completion of $C^\infty_0(\Omega)$ with the norm $\|.\|_W$, and the local versions of $W$ defined for any open set $E \subset \R^n$ as
\begin{align*}
    W_r(E):=\{u\in L^1_{loc}(E \cap \Omega), \varphi u\in W\ \text{for all $\varphi \in C_0^\infty(E)$}\}.
\end{align*}
Note that $E$ is not necessarily a subset of $\Omega$, and that $W_r(\R^n) = W^{1,2}_{loc}(\overline{\Omega}, dm) \varsubsetneq W$.

\medskip

We are now ready to talk about weak solutions to $\mathcal{L} u = 0$. Recall that $\mathcal L = - \diver (w \mathcal A \nabla)$ for a matrix $\mathcal A$ that satisfies \eqref{ELLIP}. Let $F \subset \Omega$ be an open set. We say that $u$ is a weak solution of $\mathcal{L}u=f$ in $F$ if $u\in W_r(F)$ and for any $\varphi\in C_0^\infty(F)$,
\begin{align*}
    \int_\Omega \mathcal{A}\nabla u\cdot \nabla \varphi dm=0.
\end{align*}
We can always construct a unique weak solution via the Lax-Milgram theorem.

\begin{Lemma}[Lemma 9.3 of {\cite{david2017elliptic}}]\label{THLMG}
For any $h\in W^{-1} := (W_0)^*$ and $f\in H$, there exists a unique $u\in W$ such that $\Tr (u) = f$ and 
\[\int_\Omega \mathcal A\nabla u \cdot \nabla \varphi \, dm = \left< h, \varphi\right>_{W^{-1},W_0}  \quad \text{ for } \varphi \in W_0.\]
Moreover, there exists $C>0$ independent of $h$ and $f$ such that:
\begin{align*}
    \|u\|_{W} \leq C(\|f\|_{H}+\|h\|_{W^{-1}}).
\end{align*}
\end{Lemma}

Let us now recall several classical results (Caccioppoli inequality, Moser estimate, and Harnack inequality inside the domain) that will be useful later. Since they are interior results, that is where the weight $w$ has no degeneracy, they are direct consequences of the classical theory. The precise statements can be found in \cite{david2017elliptic} and \cite{david2020elliptic}.

\begin{Lemma}[Interior Caccioppoli inequality and Moser estimate]\label{LMICAE}
Let $B$ be a ball of radius $r>0$ such that $2 B\subset \Omega$ and let $u\in W_r(2 B)$ be a weak solution to $\mathcal L u = 0$ in $2B$, then
\begin{align}\label{icaeq}
    \int_B|\nabla u|^2 dm\leq Cr^{-2}\int_{2B}u^2dm,
\end{align}
and
\begin{align}\label{icaeq222}
    \sup_B |u| \leq C \fint_{2B} |u|\, dm,
\end{align}
where $C>0$ depends on dimension $d, n$, and the elliptic constant $\lambda$.
\end{Lemma}

The interior Caccioppoli inequality (and the Moser estimate) holds if we replace $2B$ by $\alpha B$ in (\ref{icaeq}), and the constant $C$ will then depend on $\alpha>1$ too. Note also that we can very well replace a ball by a (Whitney) cube, that is a cube $I \in \R^n$ for which $2I \subset \Omega$, and that we can replace $dm$ by $dX$ since the weight $w$ is non-degenerated on $B$.

\begin{Cor} \label{ReverseHolder}
Let $B$ be a ball of radius $r>0$ such that $4B\subset \Omega$ and let $u\in W_r(2 B)$ be a weak solution to $\mathcal Lu = 0$ in $2B$, then
\[\left( \fint_B |\nabla u|^2 \, dX\right)^\frac12 \leq C \fint_{2B} |\nabla u| \, dX,\]
where $C>0$ depends on dimension $d, n$, and the elliptic constant $\lambda$.
\end{Cor}

\bp 
First, observe that $w(X) \approx w(Y)$ for $X,Y\in 2B$, that is $\fint_E v \, dm \approx \fint_E v\, dX$ whenever $v$ is nonnegative and $E \subset 2B$. Therefore, if $u_{2B} = \fint_{2B} u \, dX$, then
\[\begin{split}
\left(\fint_B|\nabla u|^2 dX\right)^\frac12 & \approx \left(\fint_B|\nabla(u-u_{2B})|^2 dX\right)^\frac12 \lesssim \frac1r \left( \fint_{\frac32 B} |u-u_{2B}|^2 dm\right)^\frac12\\
&  \lesssim \frac1r  \fint_{2B} |u-u_{2B}| dm \lesssim \frac1 r \fint_{2B} |u-u_{2B}| \, dX,
\end{split}\]
where we invoke successively \eqref{icaeq} and \eqref{icaeq222} and use the fact that we can replace $2B$ by $\alpha B$ in those inequalities. The lemma follows then from the $L^1$-Poincar\'e inequality.
\ep

\begin{Lemma}[Harnack inequality]\label{LMHANK}
Let $B$ be a ball such that $2B\subset \Omega$, and let $u\in W_r(2B)$ be a non-negative solution to $\mathcal Lu=0$ in $2B$. Then
\begin{align*}
    \sup_B u\leq C \inf_B u,
\end{align*}
where $C$ depends on dimension $d,n$, and elliptic constant $\lambda$. 
\end{Lemma}

We also have a version of  Lemma \ref{LMICAE} for a ball $B$ centered at the boundary, and in this case, the solution $u\in W_r(2B)$ has to satisfy $\Tr u = 0$ on $2B \cap \partial \Omega$. In order to keep our article short, we will not present the result explicitly, but it is worthwhile to mention the H\"older continuity of solutions at the boundary. 

\begin{Lemma}[H\"older continuity at the boundary; Lemmas 11.32 and 15.14 in \cite{david2020elliptic}]\label{LMHIB}
Let $B:=B(x,r)$ be a ball with a center $x \in \partial \Omega$ and radius $r>0$, and let $X$ be a corkscrew point associated to $(x,r/2)$. For any non-negative solution $u\in W_r(B)$ to $\mathcal Lu = 0$ in $B \cap \Omega$ such that $\Tr u = 0$ on $B \cap \partial \Omega$. There exists $\alpha>0$ such that for $0<s<r$,
\begin{align*}
    \sup_{B(x,s)} u \leq C\Big (\frac{s}{r}\Big )^\alpha u(X),
\end{align*}
where the constants $\alpha,C$ depend on dimension $n$, the uniform constants of $\Omega$, and elliptic constant $\lambda$. 
\end{Lemma}

We shall mention quickly that weak solutions are also H\"older continuous inside the domain, and so the solutions $u$ that verify the assumptions of Lemma \ref{LMHIB} are H\"older continuous in $\frac12B \cap \overline{\Omega}$.

\subsection{Elliptic measure} \label{SSharmonic}
For solutions $u\in W$ to $\mathcal Lu=0$, we have a maximum principle that states
\begin{equation} \label{MP}
\sup_\Omega u \leq \sup_{\partial \Omega} \Tr(u) \quad \text{ and } \quad \inf_\Omega u \geq \inf_{\partial \Omega} \Tr(u),
\end{equation}
see Lemma 12.8 in \cite{david2020elliptic}. The maximum principle and Riesz representation theorem can be used to construct a family of positive regular Borel measures $\omega^X$ on $\partial \Omega$, which is called the elliptic measure. 

\begin{Prop}[Elliptic measure, Lemmas 12.13 and 12.15 in \cite{david2020elliptic}]\label{DEFHM}
There exists a unique collection of Borel regular probability measures $\{\omega^X\}_{X\in \Omega}$ on $\partial \Omega$ such that, for any continuous compactly supported function $f\in H$, the solution $u_f$ constructed as 
\begin{align} \label{defhm1}
    u_f(X):=\int_{\partial \Omega} f(x)d\omega^{X}(x)
\end{align}
is the solution to $\mathcal Lu = 0$ and $\Tr u_f = f$ given by Lemma \ref{THLMG}. 

Moreover, the construction \eqref{defhm1} can be extended to all bounded functions on $\partial \Omega$ and provides a weak solution to $\mathcal Lu=0$. 
\end{Prop}

Since the elliptic measure is a family of measures, the classical definitions of $A_\infty$ and $RH_p$ should be adapted to fit the scenario of elliptic measure.
  
\begin{Def}[$A_\infty$ and $RH_p$ for elliptic measure]\label{def.arhp}
We say that $\{\omega^X\}_{X\in \Omega}$ is of class $A_\infty$ with respect to the measure $\sigma$, or simply $\{\omega^X\}_{X\in \Omega} \in A_\infty(\sigma)$, if for every $\epsilon>0$, there exists $\xi=\xi(\epsilon)>0$ such that for any boundary ball $\Delta=\Delta(x,r)$ and any $E\subset \Delta$, we have
\begin{align*}
    \frac{\sigma(E)}{\sigma(\Delta')}<\xi\implies \omega^{X_0}(E)< \epsilon,
\end{align*}
where $X_0$ is a corkscrew point associated to $\Delta$.

We say that $\{\omega^X\}_{X\in \Omega} \in RH_p(\sigma)$, for some $p\in (1,\infty)$, if there exists a constant $C\geq 1$ such that for each surface ball $\Delta$ with corkscrew point $X_0\in \Omega$, we have 
\begin{align}\label{RHP.eq00}
    \Big (\frac{1}{\sigma(\Delta)}\int_{\Delta}(k^{X_0})^p d\sigma\Big )^{1/p}\leq C\frac{1}{\sigma(\Delta)}\int_{\Delta}k^{X_0} d\sigma.
\end{align}
\end{Def}

\medskip

Let us recall a result from \cite{mayboroda2020carleson} showing that in higher co-dimension, the solvability of Dirichlet problem in $L^{p'}$ is equivalent to the fact that $\omega^X \in RH_p(\sigma)$. It is an analogue of Theorem 1.7.3 of \cite{kenig1994harmonic}.

\begin{Th}\label{DAIEUI}
Let $\omega^X$ be the elliptic measure associated to $\mathcal{L}$, and let $p,p'\in (1,\infty)$ such that $\frac{1}{p}+\frac{1}{p'}=1$. Then, the following statements are equivalent:
\begin{enumerate}[label=(\roman*)]
    \item The Dirichlet problem is solvable in $L^p$, that is, for each $f\in C_c(\partial \Omega)$, the solution $u_f$ constructed by \eqref{defhm1} satisfies
    \begin{align*}
        \|N(u)\|_{L^{p'}(\partial \Omega,\sigma)}\leq C\|f\|_{L^{p'}(\partial \Omega,\sigma)},
    \end{align*}
    where $N(u)$ is the non-tangential maximal function - see \eqref{DEFN} - and the constant $C$ is independent of $f$.
    \item We have that $\omega\ll \sigma$ and $\omega^X \in RH_p(\sigma)$ (see Definition \ref{def.arhp}).
\end{enumerate}
\end{Th}

\subsection{Green functions} \label{SSGreen}

The Green function is a function defined on $\Omega \times \Omega$ which is morally the solution to $\mathcal Lu = \delta_Y$ - where $\delta_Y$ is the Dirac distribution - with zero trace. Its properties are given below.

\begin{Th}[Lemma 14.60 and 14.91 in \cite{david2020elliptic}]\label{LMGNEX}
There exists a unique function $G:\Omega\times\Omega\rightarrow \mathbb{R}\cup \{+\infty\}$ such that $G(X,.)$ is continuous on $\Omega \setminus \{X\}$, locally integrable in $\Omega$ for any $X\in \Omega$, and such that for any $f\in C^\infty_0(\Omega)$, the function defined by
\begin{equation} \label{eqGRep}
u_f(X) := \int_\Omega G(X,Y) f(Y) dY
\end{equation}
belongs to $W_0$ and is a solution to $\mathcal Lu=f$ in the sense that 
\[\int_\Omega \mathcal A\nabla u_f \cdot \nabla \varphi \, dm = \int_\Omega f\varphi \, dm \qquad \text{ for } \varphi \in W_0.\]
Moreover, for any $Y\in \Omega$
\begin{enumerate}[label=(\roman*)]
    \item $G(.,Y)\in W_r(\mathbb{R}^n\setminus \{Y\})$ and $\Tr[G(.,Y)]= 0$.
    
    \item For $Y\in \Omega$ and $\varphi\in C_0^\infty(\Omega)$,
    \begin{align*}
        \int_\Omega\mathcal{A}\nabla_X G(X,Y)\cdot \nabla \varphi(X)dm(X)=\varphi(Y).
    \end{align*}
    In particular, $G(.,Y)$ is a solution to $\mathcal{L}u=0$ in $\Omega\setminus\{y\}$.
    
    \item For every $Y\in \Omega$, $G(.,Y)\in W^{1,2}(\Omega \setminus B_Y, dm)$ and
   \begin{align*}
      \int_{\Omega \setminus B_Y} |\nabla_X G(X,Y)|^2 dm(X) \leq C \delta(Y)^{-d},   
  \end{align*}   
  
    \item For $Y\in \Omega$ and $q\in [1,\frac{n}{n-1})$, $G(.,Y) \in W^{1,q}(2B_Y)$ and
  \begin{align*}
      \left(\fint_{2B_Y} |\nabla_X G(X,Y)|^q dm(X) \right)^\frac1q \leq C_q \delta(Y)^{-d}.   
  \end{align*}
\end{enumerate}
In the inequalities above, $C>0$ depends on $n$, the uniform constants of $\Omega$, and the ellipticity constant $\lambda$, while $C_q$ depends on the same parameters and $q$.
\end{Th}

We only provide a condensed version of Theorem 14.60 from \cite{david2020elliptic}. Indeed, we also have explicit pointwise bounds on $G$, but it turns out they are not useful in our article. So we omit them here. 

\begin{Lemma}[Lemma 10.101 of \cite{david2017elliptic}]\label{GEQGT}
Let $G_*$ be the Green function associated with the operator $\mathcal L^*$ (defined from the matrix $\mathcal{A}^T$). For any $X,Y\in\Omega, X\neq Y$, we have $G(X,Y)=G_*(Y,X)$. In particular, the function $Y\mapsto G(X,Y)$ satisfies the estimates given in Theorem \ref{LMGNEX}.
\end{Lemma}

We need a last technical lemma. 

\begin{Lemma} \label{lemtechnical}
Let $X\in \Omega$ and $\varphi \in C^\infty_0(\R^n \setminus \{X\})$, then
\[\int_\Omega \A^T \nabla G(X,Y) \cdot \nabla \varphi(Y) \, dm(Y) = - \int_{\partial \Omega} \varphi(y) \, d\omega^X(y).\]
\end{Lemma}

\bp
Take $\rho< \delta(X)/2$ such that $B(X,\rho) \cap \supp \varphi = \emptyset$. Construct $G^\rho_*(.,X)$ to be the function in $W_0$ that satisfies 
\begin{equation} \label{defGrho99}
\int_\Omega \A^T \nabla G^\rho_*(Y,X) \cdot \nabla \phi(Y) \, dm(Y) = \fint_{B(X,\rho)} \phi \, dm \qquad \text{ for } \phi \in W_0
\end{equation}
as given by Lemma \ref{THLMG} and which was constructed in \cite[Section 14]{david2020elliptic}. As shown in the proof of Theorem 14.60 from \cite{david2020elliptic}, we have that $G^{\rho_\eta}_*(.,X)$ converges to $G_*(.,X) = G(X,.)$ uniformly on compact sets of $\overline{\Omega} \setminus \{X\}$ for a subsequence $\rho_\eta \to 0$, and by the Caccioppoli inequality, we also have that $\nabla G^{\rho_\eta}_*(.,X)$ converges to $\nabla G_*(.,X) = \nabla G(X,.)$ in $L^2_{loc} (\overline{\Omega} \setminus \{X\})$.

Let now $u_\varphi$ be the weak solution in $W$ to $\mathcal Lu_\varphi=0$ in $\Omega$ with $\Tr u_\varphi = \varphi$ given by Lemma \ref{THLMG} . By Proposition \ref{DEFHM}, we have that
\[\int_{\partial \Omega} \varphi(y) \, d\omega^X(y) = u_\varphi(X).\]
Since $u_\varphi\in W_0$ is a weak solution to $\mathcal Lu_\varphi=0$, we have 
\begin{equation} \label{defGrho98}\begin{split}
\int_\Omega \A^T \nabla G^\rho_*(.,X) \cdot \nabla \varphi \, dm 
& = \int_\Omega \A^T \nabla G^\rho_*(.,X) \cdot \nabla [\varphi-u_\varphi] \, dm \\
& = \fint_{B(X,\rho)} [\varphi - u_\varphi]  \, dm  =  - \fint_{B(X,\rho)}  u_\varphi  \, dm 
\end{split}\end{equation}
by \eqref{defGrho99} and the fact that $\varphi \equiv 0$ on $B(X,\rho)$.  As previously mentioned, we have the convergence $\nabla G^{\rho_\eta}_*(.,X) \to \nabla G(X,.)$ in $L^2(\supp \varphi,dm)$, but we also have $\fint_{B(X,\rho)}  u_\varphi  \, dm \to u_\varphi(X)$ because $u_\varphi$ is a solution, hence is continuous. The lemma follows from taking the convergence $\rho \to 0$ in \eqref{defGrho98}.
\ep

\subsection{The comparison principle} \label{SScomparison}

\begin{Th}[Lemma 15.28 of {\cite{david2020elliptic}}] \label{LMCPE}
Let $x\in \partial \Omega$ and $r>0$, and let $X$ be a corkscrew point associated to $x$ and $r$. There exists a constant $C>1$ depending on $n$, $d$, the uniform constants of $\Omega$, and the elliptic constant $\lambda$ such that for $Y\in \Omega\setminus B(x,2r)$,
\begin{align}\label{LMCPE.eq01}
    C^{-1}r^{d-1}G(Y,X)\leq \omega^Y(\Delta(x,r))\leq Cr^{d-1}G(Y,X).
\end{align}
\end{Th}

The next result in line should be the fact that the elliptic measure $\omega^X$ is doubling. We need the doubling property for the elliptic measure with pole at infinity constructed in Section \ref{Sinfty}, but we shall prove this result without going through the fact that the elliptic measure is itself doubling.

At this point, it is time to introduce the comparison principle. There are two different versions of it.

\begin{Lemma}[Change of poles; Lemma 15.61 of \cite{david2020elliptic}]\label{LMCPHM}
Let $\Delta := \Delta(x,r)$ be a boundary ball, and let $X$ be a corkscrew point associated to $\Delta$. If $E \subset \Delta$ is a Borel set, then for $Y\in \Omega\setminus B(x,2r)$,
\begin{align}\label{CPHM.eq00}
    C^{-1} \omega^{X}(E) \leq \frac{\omega^{Y}(E)}{\omega^{Y}(\Delta)}\leq C\omega^{X}(E),
\end{align}
where $C>0$ depends on $n$, the uniform constants of $\Omega$, and the ellipticity constant $\lambda$.
\end{Lemma}

\begin{Th}[Comparison principle, Lemma 15.64 of {\cite{david2020elliptic}}] \label{THCPMLD}
Let $x\in \partial \Omega$ and $r>0$ be given, and take $X$ a corkscrew point associated to $x$ and $r$. Let $u,v\in W_r(B(x,2r))$ be two non-negative, not identically zero, solutions of $\mathcal{L} u=\mathcal{L}v=0$ in $B(x,2r) \cap \Omega$ such that $\Tr u=\Tr v=0$ on $\Delta(x,2r)$. For any $Y \in \Omega\cap B(x,r)$, one has
\begin{align*}
    C^{-1}\frac{u(X)}{v(X)}\leq \frac{u(Y)}{v(Y)}\leq C\frac{u(X)}{v(X)},
\end{align*}
where $C>0$ depends only on $n$, the uniform constants of $\Omega$, and the ellipticity constant $\lambda$. 
\end{Th}

The next corollary is a generalization of Corollary 6.4 of \cite{david2018square}. Even though Corollary 6.4 in  \cite{david2018square} is proved for a specific operator $\mathcal{L}$, the proof of it can be adapted to any uniformly elliptic operator, because it is a direct consequence of Theorem \ref{THCPMLD}. 

\begin{Cor}[Corollary 6.4 of \cite{david2018square}]\label{QHCB}
Under the same assumptions on $u$ and $v$ as in Theorem \ref{THCPMLD}, for all $X,Y\in B(x,\rho)\cap \Omega$ and $0<\rho<r/4$,
\begin{align*}
    \Big |\frac{u(X)}{u(Y)}\frac{v(Y)}{v(X)}-1\Big |\leq C\Big (\frac{\rho}{r}\Big )^{\alpha},
\end{align*}
where $\alpha>0$ and $C>0$ depend also only on $n$, the uniform constants of $\Omega$, and the ellipticity constant $\lambda$. 
\end{Cor}

\section{The Green function and elliptic measure with pole at infinity} \label{Sinfty}

The elliptic measure, contrary to what its name suggests, is a collection of measures. This is pretty inconvenient: every time when we consider the elliptic measure and its property, we have to pick a right one from the collection. It would be way more practical to have a single measure $\omega$ that will capture the (interesting) behavior of all the measures $\{\omega^X\}_{X\in \Omega}$. 
As we can see in \eqref{CPHM.eq00}, taking a pole $Y$ further away from the boundary set $E$ will not really matter, as long as we rescale accordingly.
For bounded domains $\Omega$, it suffices to pick a pole $X_\Omega$ which is roughly at the middle of the domain in order to have a measure $\omega:= \omega^{X_\Omega}$ from which we can recover many properties that the collection $\{\omega^X\}_{X\in \Omega}$ possess. 
For unbounded domains, we want to morally take ``$X_\Omega = \infty$''. This section is devoted to the construction of the measure $\omega^\infty$ - called the elliptic measure with pole at infinity - and its Green counterpart $G^\infty_*$, which satisfies \eqref{LMCPE.eq01} with ``$Y=\infty$''.

\begin{Def}\label{DEFGI}
We say $G_*^\infty$ and $\omega^\infty$ are the Green function and elliptic measure with pole at infinity\footnote{If we want to be accurate, the elliptic measure with pole at infinity is for $\mathcal L$ while the Green function with pole at infinity is for its adjoint $\mathcal L^*$. Indeed, we want to take $Y=\infty$ in $G(Y,X) = G_*(X,Y)$, that is when the pole of the Green function associated to $\mathcal L^*$ is $\infty$.} if $G^\infty_*\in W_r(\R^n)$\footnote{$W_r(\R^n)$ is the set $W^{1,2}_{loc}(\overline{\Omega},dm)$.} is a positive solution to $\mathcal{L}^* G^\infty_*=0$ in $\Omega$ with zero trace, and $\omega^\infty$ verifies
\begin{align*}
    \int_\Omega \mathcal{A}^T\nabla G^\infty_*\cdot \nabla \varphi \, dm= - \int_{\partial \Omega} \varphi \, d\omega^\infty,\quad \text{ for } \varphi\in C_0^\infty(\mathbb{R}^n).
\end{align*}
\end{Def}

\begin{Lemma}[Existence and uniqueness of $G^\infty$ and $\omega^\infty$]\label{LMExAUn}
Let $\Omega$ be a uniform domain and let $\mathcal{L}$ be an operator that satisfies \ref{ELLIP}. There exist a Green function and an elliptic measure with pole at infinity, and they are both unique up to multiplication by a positive scalar. 
\end{Lemma}

\begin{proof}
The proof of the following lemma is adapted from Lemma 6.5 of \cite{david2018square}. One key difference is that we consider a general operator $\mathcal{L}$, which is not necessarily symmetric. 

\medskip

Choose a boundary point $x \in \partial \Omega$ (the choice is not important). Pick $X_0 \in B(x,1)\cap \Omega$ a corkscrew point associated to $x$ and $1$, and then for $i\geq 1$, pick $X_i \in \Omega \setminus B(x,2^i)$ to be a corkscrew point for $x$ and $C2^i$. For $i\geq 1$, we define $G^i_*(X):=\frac{G(X_i,X)}{G(X_i,X_0)}$,
where $G(.,X)$ is the Green function of $\mathcal{L}$. Thanks to the Harnack inequality (Lemma \ref{LMHANK}), $G(X_i,X_1)>0$ for $i>1$. So $G^i(X)$ is well defined. 

First, we show the existence of the Green function with pole at infinity. Let $B_j:= B(x,2^j)$. Observe that for $j<i$, one has $X_i \notin 2B_j$ , so in particular $G^i_*$ is a solution in $2B_j \cap \Omega$ and hence is H\"older continuous on $B_j \cap \overline{\Omega}$ (see Lemma \ref{LMHIB}). Using the H\"older continuity, the Harnack inequality (Lemma \ref{LMHANK}), the existence of Harnack chain (by assumption on the domain), and the fact that $G^i_*(X_0) = 1$ for all $i$, we also deduce that the $G^i_*$ are uniformly bounded on $B_j \cap \Omega$. 
Thus, the sequence $\{G^i_*\}_{i >j}$ is uniformly bounded and uniformly equicontinuous (follows from the H\"older continuity), and by the Arzel\`a–Ascoli theorem, there exists a subsequence $\{G^{i_\eta}_*\}$ that converges uniformly on $B_j \cap \Omega$. By a diagonal process, we conclude that $G^{i_\eta}_*$ converges uniformly on all compact sets of $\overline{\Omega}$ to a non-negative continuous function $G^\infty_*(X)$ satisfying $G^\infty_*(X_0) = 1$.

Using the boundary Caccioppoli inequality (Lemma 11.15 in \cite{david2020elliptic}, analogous of Lemma \ref{LMICAE} but at the boundary), we can see that $\nabla G^{i_\eta}$ is a Cauchy sequence in $L^2(K)$ for all compact set $K\Subset \Omega$, and thus $\nabla G^{i_\eta}_*$ converges to a function $V\in L^2_{loc}(\overline{\Omega},dm)$. Since $\nabla G^{i_\eta}_*$ converges to both $\nabla G^\infty_*$ and $V$ in the sense of distributions, we deduce that $\nabla G^\infty_* = V$, hence $G^\infty_* \in W_r(\R^n)$.

From the previous paragraph, $\nabla G^{i_\eta}_*$ converges strongly (hence weakly) to $\nabla G^\infty_*$ in $L^2_{loc}(\Omega)$, so we easily have 
\begin{align}\label{lmgheu.eq05}
   \int_\Omega\mathcal{A}^T \nabla G^\infty_*\cdot \nabla \varphi \, dm =  - \lim_{i\rightarrow \infty}\int_\Omega \mathcal{A}^T\nabla G^i_*\cdot \nabla \varphi \, dm=0.
\end{align}
whenever $\varphi \in C^\infty_0(\Omega)$ and $i$ large enough so that $X_i$ is outside of $\supp \varphi$. We deduce that $G^\infty_*$ is a weak solution to $\mathcal L^* G^\infty_* = 0$. We can now invoke the Harnack inequality (together with the existence of Harnack chains and the fact that $G^\infty_*(X_0) = 1$)  to obtain that $G^\infty_*$ is positive in $\Omega$. 

\medskip

We claim that $G^\infty$ is the unique positive solution to the operator $\mathcal{L}^*$ 
with zero trace (up to a positive scalar multiplication).
Take another weak solution $v\in W(\R^n)$ to $\mathcal L^*v = 0$ in $\Omega$ with zero trace and $v(X_0) = 1$. Applying Corollary \ref{QHCB} with $Y=X_0$,  one has
\begin{align}\label{CPGV}
    \Big |\frac{G^\infty_*(X)}{v(X)}-1\Big |\leq C\Big (\frac{\rho}{r}\Big )^{\alpha}.
\end{align}
whenever $X \in \Omega$, $B(x,\rho) \ni X$, and $r>4\rho$. There is no limitation to take $\rho/r$ as small as we want, hence (\ref{CPGV}) implies that $G^\infty_* \equiv v$. The uniqueness also proves that, in the Arzel\`a-Ascoli theorem, $G^\infty_*$ is the only adherent point 
of the relatively compact sequence $\{G^i_*\}$. So we actually have that
\begin{equation} \label{GitoG}
G^i_* := \frac{G(X_i,.)}{G(X_i,X_0)} \text{ converges to } G^\infty_* \text{ uniformly on compact subsets of $\overline{\Omega}$ and in $W_r(\R^n)$.}
\end{equation}

\medskip

Now, we deal with the elliptic measure $\omega^\infty$ with pole at infinity. Set $\omega^i = \frac{\omega^{X_i}}{G(X_i,X_0)}$.
Theorem \ref{LMCPE} entails, for $i>j$, that  $\omega^{X_i}(B_j \cap \partial \Omega)\lesssim 2^{j(d-1)}G(X_i,X^j)$, hence 
\[\omega^i (B_j \cap \partial \Omega) \leq C 2^{j(d-1)}G^i_*(X_j) \leq C_j,\]
because $G^i_*$ converges to $G^\infty_*$ uniformly on compacts.
Thus, there exists a measure $\omega^\infty$ such that a subsequence $\omega^{i_\eta}$ converges weakly-* to $\omega^\infty$. 
Lemma \ref{lemtechnical},  the convergence of $\nabla G^i_*$ to $\nabla G$ in $L^2(\overline{\Omega},dm)$, and $\omega^{i_\eta} \stackrel{*}{\rightharpoonup}\omega^\infty$ all together imply that
\begin{align}\label{lmgheu.eq06}
    \int_\Omega \mathcal{A}^T\nabla G^\infty\cdot \nabla \varphi dm= - \int_{\partial \Omega} \varphi d\omega^\infty.
\end{align}
The uniqueness of $\omega^\infty$ follows from the uniqueness of $G^\infty$ and  $(\ref{lmgheu.eq06})$. Moreover, the uniqueness also shows the convergence of  $\omega^i$ (instead of a subsequence), that is
\begin{equation} \label{omitoom}
\omega^{i} \stackrel{*}{\rightharpoonup}\omega^\infty \quad \text{ and } \quad \omega^{i}(E) \to \omega^\infty(E) \quad \text{ for any Borel set $E\subset \partial \Omega$.}
\end{equation}
The lemma follows.
\end{proof}

The Green function and elliptic measure with pole at infinity satisfy the following CFMS-type estimates (see \cite{caffarelli1981boundary}). 

\begin{Lemma}\label{CPWI01}
Let $X\in \Omega$ be a corkscrew point associated to a boundary ball $\Delta_X:=\Delta(x,r)$, then
\begin{align*}
    C^{-1}r^{d-1}G^\infty_*(X)\leq \omega^\infty(\Delta_X)\leq Cr^{d-1}G_*^\infty(X).
\end{align*}
If moreover $E \subset \Delta_X$ is a Borel set, then
\[
  C^{-1}\omega^{X}(E)\leq \frac{\omega^{\infty}(E)}{\omega^{\infty}(\Delta_X)}\leq C\omega^{X}(E),
\]
At last, when $Y\in \Omega \setminus B(x,2r)$, we have
\begin{align*}
    C^{-1} G(Y,X) \leq \frac{G_*^\infty(Y)}{\omega^\infty(\Delta_X)} \leq C G(Y,X).
\end{align*} 
In each case, $C>0$ depends only on $n$, the uniform constants of $\Omega$, and the elliptic constant $\lambda$.
\end{Lemma}

\bp
Thanks to the convergences \eqref{GitoG} and \eqref{omitoom}, the first two results follow directly from the estimates of Theorems \ref{LMCPE} and \ref{LMCPHM} respectively.
The third one is an easy consequence of Theorem \ref{LMCPE} and the first two estimates. 
\ep

Let us show now that the elliptic measure with pole at infinity is doubling.

\begin{Lemma}[Doubling property of $\omega^\infty$]\label{DBPWI}
There exists $C>0$ depending only on $n$, the uniform constants of $\Omega$, and the elliptic constant $\lambda$ such that 
\begin{align*}
   \omega^\infty(2\Delta)\leq C\omega^\infty(\Delta) \quad \text{ for any boundary ball } \Delta.
\end{align*} 
\end{Lemma}

\begin{proof}
By Lemma \ref{CPWI01}, if $r_\Delta$ is the radius of $\Delta$, then
\[   \omega^\infty(2\Delta) \approx (2r_\Delta)^{d-1} G^\infty(X_{2\Delta}) \quad \text{ and } \omega^\infty(\Delta) \approx (r_\Delta)^{d-1} G^\infty(X_{\Delta}),\]
where $X_{\Delta}$ and $X_{2\Delta}$ are corkscrew points for $\Delta$ and $2\Delta$ respectively. The Lemma follows from the existence of Harnack chains (since $\Omega$ is uniform) and the Harnack inequality (Lemma \ref{LMHANK}).
\end{proof}

The measure $\omega^\infty$ is convenient, because it allows to capture the $A_\infty$-absolute continuity and the Reverse H\"older estimates for a collection of measures (see Definition \ref{def.arhp}) with a single measure (Definitions \ref{defAinfty} and \ref{defRHp}). 

\begin{Lemma}\label{COAIC}
We have
\[\{\omega^X\}_{X\in \Omega} \in A_\infty(\sigma) \Leftrightarrow \omega^\infty \in A_\infty(\sigma),\]
and
\[\{\omega^X\}_{X\in \Omega} \in RH_p(\sigma) \Leftrightarrow \omega^\infty \in RH_p(\sigma).\]
\end{Lemma}

\begin{proof}
The change of pole estimate (the second one) in Lemma \ref{CPWI01}, Definition \ref{def.arhp}, and Theorem \ref{DAIEUI} easily give the results.
\end{proof}

\begin{Cor}\label{RVHWI} 
Let $\Omega$ is a uniform domain and $\sigma$ be as in \eqref{DEFSIG},
and let $\mathcal{L}$ be the elliptic operator that satisfies (\ref{ELLIP}). Write $\omega^\infty$ for the elliptic measure with pole at infinity of $\mathcal{L}$. For any fixed $p\in (1,\infty)$, the following two statements are equivalent:
\begin{itemize}
\item the Dirichlet problem of operator $\mathcal{L}$ is solvable in $L^{p'}$, that is, for each $f\in C_c({\partial \Omega})$, the solution $u_f$ constructed by \eqref{defhm1} satisfies
    \begin{align*}
        \|N(u_f)\|_{L^{p'}(\partial \Omega,\sigma)}\leq C\|f\|_{L^{p'}(\partial \Omega,\sigma)},
    \end{align*}
    where $C$ is independent of $f$.
\item $\omega^\infty\ll\sigma$ and $\omega^\infty \in RH_p(\sigma)$.
\end{itemize}
\end{Cor}

\section{The Proof of Theorem \ref{thm:mn1}}  \label{Sproof}

We recall that we write $\delta(X)$ for $\dist(X,\partial \Omega)$, $w(X)$ for $\delta(X)^{d+1-n}$, $dm(X) = w(X) dX$, and $B_X$ for $B(X,\delta(X)/4)$.

In this section, $\mathcal L_0$ and $\mathcal L_1$ are two operators in the form $- \diver[w\mathcal A_i \nabla]$ that satisfy \eqref{ELLIP}. Since we assume that the Dirichlet problem for $\mathcal L_1^*$ is solvable in $L^{q'}(\sigma)$, by Corollary \ref{RVHWI}, the elliptic measure $\omega_{1,*}^\infty$ with pole at infinity satisfies the reverse H\"older bounds 
\begin{equation} \label{RHom1}
\left( \fint_\Delta \left|\frac{d\omega_{1,*}^\infty}{d\sigma}\right|^q \, d\sigma \right)^\frac1q \leq C_{q} \frac{\omega_{1,*}^\infty(\Delta)}{\sigma(\Delta)} \quad \text{ for any boundary ball } \Delta.
\end{equation}
The notations $u_0$ and $u_1$ are reserved for solutions to  $\mathcal L_0 u_0 = 0$ and $\mathcal L_1 u_1 = 1$ that satisfy the same trace condition $\Tr u_0 = \Tr u_1 = f \in C_c(\partial \Omega) \cap H$. We shall use the quantity $F(X)$ defined as:
\begin{equation} \label{DEFF}
F(X) = \int_\Omega \nabla_Y G_1(X,Y) \cdot \mathcal E(Y) \nabla u_0(Y) dm(Y) = u_1(X) - u_0(X),
\end{equation}
where $G_1$ is the Green function associated to the operator $\mathcal L_1$ and $\mathcal E:=\A_0-\A_1$ is the disagreement between $\mathcal L_0$ and $\mathcal L_1$. One important fact is that $F$ is the difference of $u_1$ and $u_0$, that is 
\begin{align}\label{THMN4.eq01}
    u_1(X) -u_0(X) =F(X) \qquad \text{ for almost every } X\in \Omega.
\end{align}
Indeed, we ``morally'' have that 
\begin{align*}
\mathcal{L}_1(u_1-u_0)= (\mathcal L_0 - \mathcal L_1) u_0 =-\diver \big (w\mathcal{E}\nabla u_0\big ),
\end{align*}
then using the properties of the Green function and the fact that $u_1 - u_0$ has zero trace, 
 \begin{multline*}
u_1(X)-u_0(X) = - \int_\Omega G_1(X,Y) \diver \big (w\mathcal{E}\nabla u_0\big )(Y) \, dY \\
= \int_\Omega \nabla_Y G_1(X,Y)\cdot \mathcal{E}(Y)\nabla u_0(Y)\, w(Y) dY = F(X).
\end{multline*}
The actual proof of \eqref{THMN4.eq01} can be found in Lemma 3.18 from \cite{cavero2019perturbations} (for codimension 1) and Lemma 7.13 from \cite{mayboroda2020carleson} (for higher codimension).

We assume that the disagreement verifies the Carleson measure condition
\begin{equation} \label{CMforE}
\int_{B(x,r) \cap \Omega} \sup_{Y\in B_X} |\E(Y)|^2 \, \frac{dX}{\delta(X)^{n-d}} \leq M r^d \qquad \text{ for any $x\in \partial \Omega$ and $r>0$}. 
\end{equation}
The condition \eqref{CMforE} is well adapted to the non-tangential maximal function $\wt N$ because of the Carleson inequality
\begin{equation} \label{Carleson}
\int_{\partial \Omega} \left(\int_{\gamma(x)} |\E(Y)|^2 |\phi(Y)|^2 \frac{dY}{\delta(Y)^{n}} \right)^\frac q2 d\sigma(x) \lesssim M \|\wt N(\phi)\|_{L^q}^q,
\end{equation}
which is proved as Lemma 2.1 in \cite{cohn2000factorization} in the case where the boundary is flat (but the proof easily extends to our setting).

\bigskip

The plan of the proof is as follows:
\begin{enumerate}
\item Lemma \ref{LE21}:
\[\|\wt N(F)\|_q \lesssim \int_\Omega \nabla F(Z) \cdot \vec h(Z) \, dZ,\]
where $\vec h$ is constructed by duality to have the above estimate (and so depends on $F$ and $q$).

\item Lemma \ref{cor22}:
\[\int_\Omega \nabla F(Z) \cdot \vec h(Z) \, dZ \lesssim M \|\wt N(\nabla u_0)\|_q \|S(v)\|_{q'} ,\]
where $v$ is the solution to $\mathcal L_1^* v = \diver \vec h$ with $\Tr(v) = 0$.

\item Lemma \ref{LM25} and Corollary \ref{CO26}:
\[ \|S(v)\|_{q'} \lesssim \|N(v)\|_{q'} +  \|\wt N(\delta \nabla v)\|_{q'} + \|\mathcal M_\omega(T(\vec h))\|_{L^{q'}(\partial \Omega,\sigma)},\]
where $\mathcal M_\omega$ is the Hardy-Littlewood maximal function with respect to the measure $\omega := \omega_{1,*}^\infty$, 
and where $T(\vec h)$ is defined in \eqref{DEFTP} and looks a bit like a square functional. 

\item Lemma \ref{LM23}:
\[\|N(v)\|_{q'} +  \|\wt N(\delta \nabla v)\|_{q'}  \lesssim \| \mathcal M_\omega(T(\vec h))\|_{L^{q'}(\partial \Omega,\sigma)}.\]

\item By the property (v) of Theorem \ref{THCABO}, 
\[\| \mathcal M_\omega(T(\vec h))\|_{L^{q'}(\partial \Omega,\sigma)} \lesssim \|T(\vec h)\|_{L^{q'}(\partial \Omega,\sigma)}.\]

\item Lemma \ref{LM27}:
\[\|T(\vec h)\|_{q'} \lesssim 1.\]

\item Items (1) to (6) prove that, for $f\in C_c(\partial \Omega) \cap H$
\begin{align} \label{main1z}
    \|\wt N(\nabla F)\|_{L^q(\partial \Omega,\sigma)}\leq C \|\wt N(\nabla u_0)\|_{L^q(\partial \Omega,\sigma)},
\end{align}
that is, by \eqref{THMN4.eq01},
\begin{align} \label{main1t}
    \|\wt N(\nabla u_1)\|_{L^q(\partial \Omega,\sigma)}\leq C \|\wt N(\nabla u_0)\|_{L^q(\partial \Omega,\sigma)},
\end{align}
which is the desired result. 
\end{enumerate}

The constants in this section are independent of $f$ and depend on $\mathcal L_0$, $\mathcal L_1$ (and hence  $u_0$, $u_1$) only via the ellipticity constant $\lambda$ and the reverse H\"older constants $q$ and $C_q$ from \eqref{RHom1}. The dependence in $M$ will only appear in Lemma \ref{cor22} and will be explicitly written.
\subsection{Notation} \label{SSWhitney}

We start this section by giving the definition of cones that we shall use. The basic cones are simply $\gamma(x) := \{X\in \Omega, \, |X-x| < 2\delta(X) \}$, but it will be also convenient for us to use cones constructed from Whitney cubes.

So we construct a family of Whitney cubes $\mathcal W$. We use the following convention: if $Q\in \mathbb D$ is a dyadic cube in $\R^n$, then $\ell(Q)$ denotes its diameter and 
\[\kappa Q := \{X\in \R^n, \, \dist(X,Q) \leq (\kappa-1)\ell(Q)\} \qquad \text{ for } \kappa \geq 1,\]
in particular if $Q^*$ is the dyadic parent of $Q$, then $Q^* \subset 2Q$.
We say that the dyadic $I \in \mathbb D$ in $\R^n$ belongs to $\mathcal W$ if $I$ is a maximal dyadic cube with the property that $10I \subset \Omega$. As such, a cube $I\in \mathcal I$ satisfies 
\begin{equation} \label{defIxiI}
10 I \subset \Omega, \qquad 20I \cap \partial \Omega\neq \emptyset.
\end{equation}
We define then $\gamma_d(x)$ as the union of the Whitney cubes that intersect $\{X\in \Omega, \, |X-x| < 3\delta(X) \}$, that is, if
\[\mathcal W_x := \{I \in \mathcal W, \, |X-x| < 3\delta(X) \text{ for one } X\in  I\},\]
then
\begin{equation} \label{defgamma}
\gamma_d(x) := \bigcup_{I\in \mathcal W_x} I.
\end{equation}

\subsection{Duality and the function $h$}

The first step is to use duality to write $\|\wt N(\nabla F)\|_q$ as an integral against a function. Since we do not know {\em a priori} that $\|\wt N(\nabla F)\|_q$ is finite, for the rest of the section, we choose a compact subset $K$ of $\Omega$ and we define the truncated (localized) function $\wt N$ as
\[\wt N_K(\nabla F) := \sup_{X\in \gamma(x)} \1_K(X) \left(\fint_{B_X} |\nabla F|^2 dY\right)^\frac12.\]
The quantities $\wt N_K(\nabla F)(x)$ - for $x\in \partial \Omega$ - and $\|\wt N_K(\nabla F)\|_q$ are all finite, and this is only a consequence of the fact that $F=u_1-u_0 \in W^{1,2}_{loc}(\Omega)$. We shall obtain bounds on $\|\wt N_K(\nabla F)\|_q$ that are independent of $K$, hence a bound on $\|\wt N(\nabla F)\|_q$ thanks to the monotone convergence theorem.

\begin{Lemma}\label{LE21}
Let $q>1$ and let $K\Subset \Omega$. There exist a compact set $K'\Subset \partial \Omega$,  a bounded vector function $\vec\alpha \in L^\infty(\Omega, \R^n)$ with $\|\vec \alpha\|_{\infty} =1$, a non-negative function $\beta(.,x)\in L^1(\Omega)$ with $\int_{\Omega}\beta(X,x)dX=1$ for all $x\in{\partial \Omega}$, and a non-negative function $g\in L^{q'}(\partial \Omega)$ with $\|g\|_{L^{q'}(\partial \Omega)}=1$  such that:
\begin{align}\label{LE21.eq1}
    \|\wt N_K(\nabla F)\|_{L^q}\leq C\int_\Omega\nabla F(Z)\cdot \vec h(Z)dZ,
\end{align}
where $C$ depends only on $n$ and $\lambda$ and where $\vec h$ is defined as
\begin{align}\label{LM21.defh}
    \vec h(Z):= \int_{K}  \vec\alpha(Z) \1_{2B_X}(Z) \int_{K' \cap 8B_X} \beta(X,x) g(x)d\sigma(x) \frac{dX}{\delta(X)^n} .
\end{align}
\end{Lemma}

\begin{Rem}
The function $\vec h$, as well as the functions $g$ and $\beta$, depend on the compact $K$. It is necessary to guarantee the {\em a priori} finiteness of all $\|\wt N_K(\nabla F)\|_{L^q}$, and so of all the quantities we shall manipulate in the future. 
Moreover, the function $\vec h$ is compactly supported and bounded by a constant that depends on $K$ (see Lemma \ref{lemh}), which will make future manipulations of $\vec h$ easier.
{\bf However, the constants in the core results of this section} (Lemmas \ref{LE21}, \ref{LM27}, \ref{cor22}, \ref{LM25}, \ref{LMCPVE})   {\bf shall never depend on $K$,} so that the bound that we obtain on $\|\wt N_K(\nabla F)\|_{L^q}$ will be transmitted to $\|\wt N(\nabla F)\|_{L^q}$.
\end{Rem}

\begin{proof}
First, recall that $F$ is just $u_1-u_0$, see \eqref{THMN4.eq01}, so we can use the reverse H\"older inequality for the gradient (see Corollary \ref{ReverseHolder}) to obtain that 
\[\wt N_K(\nabla F)(x) \leq C_\Lambda \wt N^{1}_K(\nabla F)(x) \qquad \text{ for } x\in \partial \Omega,\]
where 
\[\wt N^1_K(\nabla F) := \sup_{X\in \gamma(x)} \1_K(X) \fint_{2B_X} |\nabla F| dY.\]
Of course, it also gives that $\|\wt N_K(\nabla F)\|_q \lesssim \|\wt N^1_K(\nabla F)\|_q$. 

\medskip

The rest of the proof relies on duality. Since $L^{q'}(\partial \Omega,\sigma)$ is the dual space of $L^q(\partial \Omega,\sigma)$ and $\wt N_K^1(\nabla F)$ is non-negative, we have
\[\|\wt N^1_K(\nabla F)\|_q = \sup_{\begin{subarray}{c} 0 \leq g \in L^{q'}  \\ \|g\|_{q'} =1  \end{subarray}} \int_{\partial \Omega} \wt N_K^1(\nabla F) g \, d\sigma.\]
We are able to select a $g \in L^{q'}$ with $g\geq 0$ and $\|g\|_{q'} =1$ such that 
\[ \|\wt N_K(\nabla F)\|_q \lesssim \|\wt N^1_K (\nabla F)\|_q \leq 2 \int_{\partial \Omega} \wt N_K^1(\nabla F) g \, d\sigma .\]
By density, we can even take $g$ to be continuous and compactly supported. We set $K'$ to be the support of $g$ and we obtain
\[ \|\wt N_K(\nabla F)\|_q \lesssim  \int_{K'} \wt N_K^1(\nabla F) g \, d\sigma.\]

Since $L^\infty$ is the dual of $L^1$, for each $x\in \partial \Omega$, we have
\begin{equation}\label{LNM02}
    \wt N^{1}_K(\nabla F)(x)  =\sup_{\|\beta(.,x)\|_{L^1(\Omega)}=1}\int_{K} \Big ( \fint_{2B_X}|\nabla F|dZ\Big ) \beta(X,x)\1_{\gamma(x)}(X)dX.
\end{equation}
It also means that we can find a function $\beta \geq 0$ which satisfies $\int_\Omega \beta(X,x) dX = 1$ for all $x\in \partial \Omega$  such that
\[\begin{split} 
\|\wt N_K(\nabla F)\|_q & \lesssim  \int_{K'} \wt N_K^1(\nabla F) g \, d\sigma \lesssim \int_{K'} g(x) \int_K \Big ( \fint_{2B_X}|\nabla F|dZ\Big ) \beta(X,x)\1_{\gamma(x)}(X)dX \, d\sigma(x) \\
& \lesssim \int_{K} \Big ( \int_{2B_X}|\nabla F|dZ\Big ) \Big(\int_{K' \cap 8B_X} \beta(X,x) g(x)  \, d\sigma(x) \Big) \, \frac{dX}{\delta(X)^n},
\end{split}\]
where the last line is due to Fubini's theorem, since $X\in \gamma(x)$ is equivalent to $x\in \partial \Omega \cap 8B_X$.

We now take $\vec \alpha = \frac{\nabla F}{|\nabla F|}$ and, by Fubini's theorem again, we have
\[\begin{split} 
\|\wt N_K(\nabla F)\|_q & \lesssim  \int_{K} \Big( \int_{\Omega} \nabla F(Z) \cdot \vec\alpha(Z) \1_{2B_X}(Z) dZ \Big) \Big(\int_{K' \cap 8B_X} \beta(X,x) g(x)  \, d\sigma(x) \Big) \, \frac{dX}{\delta(X)^n} \\
& \quad = \int_\Omega \nabla F(Z) \cdot \left( \int_K \int_{K' \cap 8B_X}  \vec \alpha(Z) \1_{2B_X}(Z) \beta(X,x) g(x) \, d\sigma(x) \, \frac{dX}{\delta(X)^n} \right) dZ \\
&  \quad = \int_\Omega \nabla F(Z) \cdot \vec h(Z) \, dZ.
\end{split}\]
The lemma follows. 
\end{proof}

In the previous construction, we made sure that $\vec h$ is nice enough, that is bounded and compactly supported, as shown in the next result.

\begin{Lemma}\label{lemh}
The function $\vec h$ defined in \eqref{LM21.defh} is bounded and compactly supported.
\end{Lemma}
 
 \bp
  Since $Z\in 2B_X$, we have $\delta(Z)/2 \leq \delta(X) \leq 2 \delta(Z)$. Combined with $|\vec \alpha| \leq 1$, we deduce
 \[|\vec h| \lesssim \delta(Z)^{-n}  \int_{K'}  g(x) d\sigma(x) \int_{K} \beta(X,x) dX \leq C_{K'}.\]
 So the function $\vec h$ is indeed bounded. 
 
It is also compactly supported because, in order for $\vec h$ to be non-zero, we need $Z \in 2B_X$ with $X\in K$. And that is possible only when $Z$ is in a compact set that is slightly bigger than $K$.
 \ep
 
 \medskip

We want now to bound the integral $\int_\Omega \nabla F \cdot \vec h \, dZ$. However, as one can expect, a lot of information is hidden in $\vec h$. Why do we use the quantity $\vec h$? Because, even if $\vec h$ depends on $F$ (and $K$), we are able to bound it independently of $F$ (and $K$), as shown in the lemma below. We define first $T(\vec h)$ as
\begin{equation} \label{DEFTP}
    T(\vec h)(x):=\sum_{I \in \mathcal W_x} \ell(I)^{n-d} \sup_{I} |\vec h|,
\end{equation}
where $\mathcal W_x$ is the collection of Whitney cubes that intersect $\gamma_3(x):= \{x\in \Omega, \, |X-x| < 3\delta(X)\}$ and $\ell(I)$ is the side-length of $I$, which is equivalent to $\dist(I,\partial \Omega)$. To build intuition, we observe that if the supremum was replaced by a $L^1$-average, then $T(\vec h)(x)$ would be essentially $\int_{\gamma_d(x)} |\vec h| dX/\delta(X)^d$, that is the integration of $|\vec h|$ over the radial direction(s). 

\begin{Lemma}\label{LM27}
We have
\begin{align*}
    \|T(\vec h)\|_{L^{q'}(\partial \Omega,\sigma)}\leq C_{q'},
\end{align*}
where $q$ is the one of Lemma \ref{LE21} and is used to construct $\vec h$.
\end{Lemma}

\begin{proof}
We first remove $\vec \alpha$ from the estimate on $\vec h$, because we won't be able to do anything with it, so we have
\begin{equation} \label{LM27a}\begin{split}
|\vec h(Z)| & \leq \int_{\Omega} \1_{2B_X}(Z) \int_{8B_{X}} \beta(X,x) g(x) \, d\sigma(x) \, \frac{dX}{\delta(X)^n}\\
& \quad = \int_{\partial \Omega} g(x) \left( \int_\Omega \1_{2B_X}(Z) \beta(X,x) \1_{\gamma(x)}(X) \, \frac{dX}{\delta(X)^n} \right) \, d\sigma(x)
.\end{split}\end{equation}

Pick a Whitney cube $I \in \mathcal W$. Construct $I^*$ as 
\[I^* := \{X\in \Omega, \, \text{there exists $Z\in I$ such that } Z\in 2B_X\}.\]
Check that $I^*$ is a Whitney region larger than $I$, but still has a finite overlapping. So if $b_I(x)$ denotes $\int_{I^*} \beta(X,x) dX$, we have the nice control
\begin{equation} \label{bdonbI}
\sum_{I\in \mathcal W} b_I(x) \lesssim 1 \qquad \text{ for any } x\in \partial \Omega
\end{equation}
because, by definition, $\int_\Omega \beta(X,x) dX = 1$ for any $x\in \partial \Omega$. We take now $Z\in I$, in this case, any $X$ that satisfies $Z\in 2B_X$ lies in the Whitney region $I^*$, which implies that $\delta(X) \approx \ell(I)$. So \eqref{LM27a} becomes
\[\sup_{2I} |\vec h| \leq \ell(I)^{-n} \int_{\partial \Omega \cap 10^3 I} g(y) b_I(y) d\sigma(y).\]

We inject this bound in the expression of $T(\vec h)$ to obtain
\begin{equation} \label{LM27c}
T(\vec h)(x) \leq \sum_{I\in \mathcal W_x} \ell(I)^{-d} \int_{\partial \Omega \cap 10^3I} g(y) b_I(y) \, d\sigma(y). 
\end{equation}

We compute then the $L^{q'}$-norm of $T(\vec h)$ by duality. Let $\phi \in L^{q}(\partial \Omega,\sigma)$ be any non-negative function such that $\|\phi\|_{q} = 1$. We claim that 
\begin{equation} \label{LM27b}
\int_{\partial \Omega} T(\vec h)(x) \phi(x) \, d\sigma(x) \lesssim 1,
\end{equation}
which is exactly what we need to conclude the lemma. We use the bound \eqref{LM27c} and then Fubini's theorem to write
\[\begin{split}
\int_{\partial \Omega} T(\vec h)(x) \phi(x) \, d\sigma(x) & \lesssim \int_{\partial \Omega} \phi(x) \sum_{I\in \mathcal W_x} \ell(I)^{-d} \int_{\partial \Omega \cap 10^3I} g(y) b_I(y) \, d\sigma(y) \, d\sigma(x) \\
& \lesssim \int_{\partial \Omega} g(y) \sum_{I \in \mathcal W_x} b_I(y) \1_{10^3I}(y) \ell(I)^{-d} \int_{\partial \Omega \cap 10^3I} \phi(x) \, d\sigma(x) \, d\sigma(y),
\end{split}\]
where the last line holds because $I\in \mathcal W_x$ implies $x\in 1000I$ (we are not trying to be optimal here). Let $\mathcal M_\sigma$ denote the Hardy-Littlewood maximal function with respect to the $d$-dimensional Ahlfors measure $\sigma$. Since $y\in 10^3I$, we easily have 
\[\1_{y\in 10^3} \ell(I)^{-d} \int_{\partial \Omega \cap 10^3I} \phi(x) \, d\sigma(x) \lesssim \mathcal M_\sigma(\phi)(y)\]
and hence
\[\begin{split}
\int_{\partial \Omega} T(\vec h)(x) \phi(x) \, d\sigma(x) & 
\lesssim  \int_{\partial \Omega} g(y) \sum_{I \in \mathcal W} b_I(y) \mathcal M_\sigma(\phi)(y) \, d\sigma(y) \\
& \lesssim \int_{\partial \Omega} g(y) \mathcal M_\sigma(\phi)(y) \, d\sigma(y).
\end{split}\]
by \eqref{bdonbI}. We invoke now the H\"older's inequality and the $L^q$-boundedness of the operator $\mathcal M_\sigma$ to deduce
\[\begin{split}
\int_{\partial \Omega} T(\vec h)(x) \phi(x) \, d\sigma(x) & 
\lesssim  \|g\|_{L^{q'}(\partial \Omega,\sigma)} \|\mathcal M_\sigma(\phi)\|_{L^q(\partial \Omega,\sigma)} \lesssim \|g\|_{L^{q'}(\partial \Omega,\sigma)} \|\phi\|_{L^q(\partial \Omega,\sigma)} = 1
\end{split}\]
because, by definition, $\|g\|_{q'} = 1$ and $\|\phi\|_q = 1$. The claim \eqref{LM27b} and the lemma follow.
\end{proof}

\subsection{The solution $v$}

Our goal is to bound the expression $\int_\Omega \nabla F \cdot \vec h\, dZ$ from \eqref{LE21.eq1}. However, this expression is lacking derivatives. Indeed, the techniques employed here rely on integration by parts, that is, moving gradients and derivatives from one term to another, with errors that can be controlled. 
So the more terms with derivatives we have, the more possibilities we get.  
That is why we introduce $v$, which is essentially the solution to the inhomogeneous Dirichlet problem
\begin{align}\label{cor22.eq1}
\begin{cases}
&\mathcal{L}^*_1v= \diver \vec h \ \ \text{ in }\Omega;\\
&\Tr (v)=0\ \ \text{ on $\partial \Omega$},
\end{cases}
\end{align}
where $\vec h$ is the one constructed in Lemma \ref{LE21}.

We shall ultimately use two distinct representations of $v$. So we need to prove that those two definitions of $v$ coincide, which is very classical in the bounded codimension 1 case but more delicate in our context (which allows higher codimensional boundaries, and the elliptic theory is not as developed).

We write $G_1$ for the Green function associated to $\mathcal L_1$ as defined in Theorem \ref{LMGNEX}. We define $v$ on $\Omega$ as
\begin{equation} \label{defv}
v(X):= -\int_\Omega \nabla_Z G_1(Z,X) \cdot \vec h(Z) \, dZ,
\end{equation}
which is well defined because $\vec h(Z)$ is bounded and compactly supported (see Lemma \ref{lemh}) and $\nabla_Z G_1^*(X,.) = \nabla_Z G_1(.,X) \in L^{r}_{loc}(\Omega)$ for $r$ sufficiently close to $1$ (see items (iii)-(iv) of Theorem \ref{LMGNEX}).

\begin{Lemma} \label{lemv}
The function $v(X)$ constructed in \eqref{defv} lies in $W_0$ and verifies
\[\int_\Omega \mathcal A_1^T \nabla v \cdot \nabla \varphi \, dm = -\int_\Omega \vec h \cdot \nabla \varphi \, dX \qquad \text{ for } \varphi \in W_0.\]
\end{Lemma}

\bp
The idea of the proof is: if $\vec h$ were smooth, there would be no difficulty. Thus, as expected, we mollify $\vec h$ and we check that we can take all the desired limits.

We construct the mollifier by using a non-negative radial function $\rho\in C_0^\infty(\R^n)$ supported in $B(0,1)$ and verifying $\int_{\R^n} \rho =1$, and then we define $\rho_{\epsilon}(Z):=\epsilon^{-n}\rho(\epsilon^{-1}Z)$ for $\epsilon>0$ and $Z\in \Omega$. We set $\vec h_{\epsilon}:=\vec h\ast \rho_{\epsilon}\in C_0^\infty(\Omega)$. The fact that $\vec h_\epsilon$ is compactly supported in $\Omega$ is true only for small $\epsilon>0$ (but it does not matter since we intend to take limits) because $\vec h$ is already compactly supported in the first place. 

We fix $p\in (n,\infty)$, so that $\nabla_Z G_1(Z,X)$ is locally in $L^{p'}$ (see item (iv) in Theorem \ref{LMGNEX}). Note for later that 
\begin{equation} \label{lemv1}
\vec h_{\epsilon} \rightarrow \vec h \text{ in } L^{p}(\Omega),
\end{equation}
which is a classical result and essentially equivalent to the density of smooth functions in $L^p$.

We define
\begin{equation} \label{defveps} \begin{split}
v_\epsilon(X) := -\int_\Omega \nabla_Z G_1(Z,X) \cdot \vec h_\epsilon(Z) \, dZ = \int_\Omega G_1(Z,X) \diver( \vec h_\epsilon)(Z) \, dZ. \end{split}\end{equation}
Since now $\diver( \vec h_\epsilon) \in C^\infty_0(\Omega)$, by \eqref{eqGRep} (and Lemma \ref{GEQGT}), we have that $v_\epsilon$ is the function of $W_0$ that satisfies
\begin{equation} \label{lemv2} \begin{split} \int_\Omega \mathcal A_1^T \nabla v_\epsilon \cdot \nabla \varphi \, dm & = \int_\Omega \diver(\vec h_\epsilon)  \varphi \, dX = -\int_\Omega \vec h_\epsilon \cdot \nabla \varphi \, dX \qquad \text{ for } \varphi \in W_0. 
\end{split}\end{equation}
In addition, we also know the following convergences. 
\begin{enumerate}
\item We can pass the limit as $\epsilon\to 0$ in the expression $\int_\Omega \nabla G_1(.,X) \cdot \vec h_\epsilon \, dZ$, because $\nabla G_1(.,Z)\in L^{p'}_{loc}(\Omega)$, all the $h_\epsilon$ are supported in the same compact subset of $\Omega$, and $\vec h_\epsilon$ converges to $\vec h$ in $L^p(\Omega)$ (see \eqref{lemv1}). So we deduce $v_\epsilon \to v$ pointwise (and thus in the distribution sense).
\item  The functions $\diver(\vec h_\epsilon)$ converge to $\diver(\vec h)$ in $W^{-1}$. Indeed, if $K \Subset \Omega$ is a compact set that contains the support of all the $h_\epsilon$ and $p\in (2,\infty)$, then
\[\begin{split}
 \|\diver(\vec h_\epsilon) - \diver(\vec h)\|_{W^{-1}} 
 & = \sup_{\|\varphi\|_{W_0} \leq 1} \left| \int_\Omega  (\vec h - \vec h_\epsilon) \cdot \nabla \varphi \, dX  \right| \\
 & \leq \|\vec h - \vec h_\epsilon\|_{L^p} \|\nabla \varphi\|_{L^{p'}(K)} \\
 & \leq C_K \|\vec h - \vec h_\epsilon\|_{L^p} \to 0 \quad \text{ as } \epsilon \to 0.   \end{split} \]
\item The previous convergences combined with the Lax-Milgram theorem (Lemma \ref{THLMG}) imply that $v_\epsilon$ converges in $W_0$.
\end{enumerate}
The combination of three convergences shows that 
\begin{equation} \label{lemv3}
v_\epsilon \to v \qquad \text{ in } W_0
\end{equation}
so in particular, $v\in W_0$ and we also have 
\[\int_\Omega \mathcal A_1^T \nabla v \cdot \nabla \varphi \, dm = - \int_\Omega \vec h \cdot \nabla \varphi \, dX \qquad \text{ for } \varphi \in W_0\]
by taking the limit in \eqref{lemv2}. The lemma follows.
\ep

\medskip

We return to the estimate of the non-tangential maximal function. Recall that at this point, we want to bound $\int_\Omega \nabla F \cdot \vec h\, dZ$. The next  step will involve the square function of $v$, which is defined as
\begin{equation} \label{defSv}
    S(v)(x):=\Big (\int_{\gamma(x)} |\nabla v|^2\frac{dY}{\delta(Y)^{n-2}}.\Big )^{1/2}.
\end{equation}

Even though the next lemma is an analogue of Lemma 2.8 in \cite{kenig1995neumann}, we provide here an alternative proof which is self contained (up to some basic results on the Green functions) and does not rely on taking the limit of a sequence of elliptic operators with smooth coefficients.

\begin{Lemma}\label{cor22}
Recall that $M$ is the constant in \eqref{CMforE}. We have
\[\int_\Omega \nabla F \cdot \vec h\, dZ \leq C M \|\wt N(\nabla u_0)\|_{L^q(\partial \Omega)} \|S(v)\|_{L^{q'}(\partial \Omega)},\]
where the constant depends only on the constant in \eqref{Carleson}.
Hence, thanks to Lemma \ref{LE21},
\[\|\wt N_K(\nabla F)\|_{L^q(\partial \Omega)} \lesssim M \|\wt N(\nabla u_0)\|_{L^q(\partial \Omega)} \|S(v)\|_{L^{q'}(\partial \Omega)}.\]
\end{Lemma}

\bp We claim that 
\begin{equation} \label{cor22a}
\int_\Omega \nabla F \cdot \vec h\, dZ = \int_\Omega \E \nabla u_0 \cdot \nabla v \, dm.
\end{equation}

To see how (\ref{cor22a}) proves our lemma, we first observe that for any positive function $\phi$ on $\Omega$, by Fubini's theorem,
\[\begin{split}
\int_{\partial \Omega} \left( \int_{\gamma(x)} \phi(X) \, \frac{dX}{\delta(X)^{n-1} }\right) \, d\sigma(x) 
& \gtrsim \int_\Omega \phi(X) \delta(X)^{-d} \sigma(8B_X \cap \partial \Omega) \, dm(X) \gtrsim \int_\Omega \phi(X) \, dm(X)
\end{split}\]
because, if $\hat x$ is such that $|X-\hat x| = \delta(X)$, then $\sigma(8B_X ) \geq \sigma(\Delta(\hat x,\delta(X))) \gtrsim \delta(X)^d$ by \eqref{DEFSIG}.
As a consequence, by successively applying the Cauchy-Schwarz inequality and the H\"older inequality, the claim \eqref{cor22a} implies
\[\begin{split}
\int_\Omega \nabla F \cdot \vec h\, dZ 
& \leq \int_\Omega |\E| |\nabla u_0| |\nabla v| \, dm \\
& \lesssim \int_{\partial \Omega} \left( \int_{\gamma(x)} |\E| |\nabla u_0| |\nabla v| \, \frac{dX}{\delta(X)^{1-n}}\right) \, d\sigma(x) \\
& \lesssim  \left(\int_{\partial \Omega} \left(  \int_{\gamma(x)} |\E|^2 |\nabla u_0|^2 \frac{dX}{\delta(X)^{n}} \right)^\frac q2 \, d\sigma(x)\right)^{\frac1q} \|S(v)\|_{L^{q'}(\partial \Omega}.
\end{split}\]
The lemma follows then from the Carleson inequality \eqref{Carleson}.

\medskip

So it remains to show the claim \eqref{cor22a}. Formally, the claim is just a permutation of integrals, that is, by using the definition \eqref{DEFF} of $F(X)$  and \eqref{defv}, one has 
\[\begin{split}
\int_\Omega \nabla F \cdot \vec h\, dX & = \int_\Omega F \diver(\vec h) \, dX \\
& = \int_\Omega \int_\Omega \nabla_Y G_1(X,Y)  \cdot \E(Y) \nabla u_0(Y) \diver(\vec h)(X) \, dm(Y) \, dX \\
& = \int_\Omega  \E(Y) \nabla u_0(Y) \cdot \nabla \Big( \int_\Omega G_1(X,Y) \diver(\vec h)(X) \, dX \Big)  dm(Y) \\
& = \int_\Omega \E \nabla u_0 \cdot \nabla v \, dm.
 \end{split}\]
However, the assumptions of Fubini's theorem are not verified, so the justification will end up being way more delicate.
 
The issue mainly comes from the Green function $G_1$, which has a degeneracy when $Z=Y$ that we cannot control very well. So instead, we shall use approximation of the Green function. We use the same mollifier as the previous lemma. Let $\rho\in C_0^\infty(\R^n)$ supported in $B(0,1)$ and verifying $\int_{\R^n} \rho =1$, and then define $\rho_{\eta}(Z):=\eta^{-n}\rho(\eta^{-1}Z)$ for $\eta>0$ and $Z\in \Omega$. We also construct a cut-off function $\varphi_\eta \in C^\infty_0(\Omega)$ such that $\varphi_\eta(Z)= 0$ if $\delta(Z) <2\eta$, $\varphi_\eta(Z) = 1$ if $\delta(Z) >4\eta$, and $|\nabla \varphi_\eta| \leq \eta^{-1}$.
Check that the map $\phi \in W_0 \to \rho_\eta * (\varphi_\eta\phi)(Z)$  lies in $W^{-1} = (W_0)^*$ for all $\eta>0$ and all $Z \in \Omega$. 
By the Lax-Milgram theorem (see Lemma \ref{THLMG}), for each $\eta>0$ and $X,Y\in \Omega$, we can construct\footnote{from here and forward, we drop the indice 1 on G, since any Green function will always be associated to $\mathcal L_1$ or $\mathcal L_1^*$.}  $G_\eta(.,Y)$ and $G^*_\eta(.,X)$ as the only functions in $W_0$ such that
\begin{equation} \label{cor22c}
\int_\Omega \mathcal A_1 \nabla_Z G_\eta(Z,Y) \cdot \nabla \phi(Z) \, dm(Z) = \rho_\eta * (\varphi_\eta\phi)(Y) \quad \text{ for } \phi \in W_0
\end{equation}
and similarly
\begin{equation} \label{cor22d}
\int_\Omega \mathcal A_1 \nabla \phi(Z)  \cdot  \nabla_Z G_\eta^*(Z,X) \, dm(Z) = \rho_\eta * (\varphi_\eta\phi)(X) \quad \text{ for } \phi \in W_0.
\end{equation}
The combination of the two above lines easily gives the nice identity
\begin{equation} \label{cor22e} \begin{split}
\mathcal G_\eta(X,Y)& := [\rho_\eta * (\varphi_\eta G^*_\eta)(.,X)](Y) = [\rho_\eta * (\varphi_\eta G_\eta)(.,Y)](X) \\
& \qquad = \int_\Omega \mathcal A_1 \nabla G_\eta(Z,Y) \cdot  \nabla G_\eta^*(Z,X) \, dm(Y).
\end{split} \end{equation}
Note that the identity implies that the function $\mathcal G_\eta$ lies in $W_0$ and is smooth both in $X$ and $Y$, which will makes $\mathcal G_\eta$ a nice tool for the next lines. We define $v_\epsilon$ as in \eqref{defveps} and
\[v_{\epsilon,\eta}(Y) := - \int_\Omega \nabla_X G_\eta(X,Y) \cdot \vec h_\epsilon(X) \, dX.\]
We plug in $G_\eta(.,Y)$ as the test function in \eqref{lemv2} to get
\[\begin{split}
v_{\epsilon,\eta}(Y) =  \int_\Omega \mathcal A^T_1 \nabla v_\epsilon \cdot \nabla G_\eta(.,Y) \, dm = \rho_\eta * (\varphi_\eta v_\epsilon)(Y)
\end{split}\]
by \eqref{cor22c}. By \eqref{lemv3}, we have that $v_\epsilon \to v$ in $W_0$, and a classical convolution result yields that $v_{\epsilon,\eta}$ converges to $ v_{0,\eta}:= \rho_\eta*(\varphi_\eta v)$ in $W_0$ (as $\epsilon \to 0$ and uniformly in $\eta$). The fact that $\nabla v_{0,\eta} \to \nabla v$ in $L^2$ is also well known. So by a diagonal argument, the function 
\[v_{\epsilon,\epsilon}(Y) =-\int_\Omega \nabla_X G_\epsilon(X,Y) \cdot \vec h_\epsilon(X) \, dX = \left<\diver(\vec h),\mathcal G_\epsilon(.,Y) \right>_{W^{-1},W_0}\]
converges to $v$ in $W_0$ as $\epsilon \to 0$. The same proof gives that 
\[\begin{split}
F_{\epsilon,\epsilon}(X)& := \int_\Omega G^*_\epsilon(Y,X) \varphi_\epsilon(Y) \diver(\rho_\epsilon*(\E\nabla u_0))(Y) \, dm(Y) \\
& = \left<\diver(\E\nabla u_0), \mathcal G_\epsilon(X,.) \right>_{W^{-1},W_0}
\end{split}\]
converges to $F$ in $W_0$. All these convergences show that the claim \eqref{cor22a} would be proven once we establish that
\begin{equation} \label{cor22f}
\int_\Omega \nabla F_{\epsilon,\epsilon} \cdot \vec h\, dX = \int_\Omega \E \nabla u_0 \cdot \nabla v_{\epsilon,\epsilon} \, dm
\end{equation}
for any $\epsilon>0$. Observe that if we replace $\rho_\eta$ by $\rho_\eta * \rho_\eta$ in \eqref{cor22c} and \eqref{cor22d}, then we replace $G_\eta(Z,Y)$ by $(\rho_\eta*G_\eta(Z,.))(Y)$ and $G_\eta^*(Z,X)$ by $(\rho_\eta*G^*_\eta(Z,.))(X)$. We deduce that we can make $G_\eta(Z,Y)$ and $G_\eta^*(Z,X)$ as smooth as we want in the second variable, and hence quantities like  $\nabla_Y(\nabla_X G_\eta)(X,Y)$ make perfect sense and lie in $L^\infty_Y(L^2_X)$. From these remarks and \eqref{cor22e}, the identity \eqref{cor22f} is just a permutation of integrals and two differentiations under the integral symbol.
\ep

\subsection{The $S<N$ estimate.}

The aim of this subsection is to bound $\|S(v)\|_q$ by the non-tangential maximal function of $v$, $\delta\nabla v$, and a term that depends on $T(\vec h)$ defined in \eqref{DEFTP}.

For the first time, we shall use \eqref{RHom1}, but in a weaker form (see Theorem \ref{THCABO}) which says that there exist $C,\theta >0$ such that
\begin{equation} \label{RHom2}
\frac{\omega^\infty_{1,*}(E)}{\omega^\infty_{1,*}(\Delta)} \leq C \left( \frac{\sigma(E)}{\sigma(\Delta)} \right)^\theta \quad \text{ and } \quad   \frac{\sigma(E)}{\sigma(\Delta)}  \leq C \left( \frac{\omega^\infty_{1,*}(E)}{\omega^\infty_{1,*}(\Delta)} \right)^\theta
\end{equation}
for any boundary ball $\Delta \subset \partial \Omega$ and any Borel set $E \subset \Delta$.

In the next lemma, $\mathcal M_\omega$ and $\mathcal M_\sigma$ are the Hardy-Littlewood maximal function for the elliptic measure $\omega := \omega^\infty_{1,*}$ and the Ahlfors regular measure $\sigma$. Moreover, $S_*$ and $\wt N_*$ are respectively the square function and  the averaged non-tangential maximal function, but defined with a wider cone $\gamma^*(x) = \{X\in \Omega, \, |X-x| < C_*\delta(X)$\}. The value $C_*$ of the aperture does not matter much, and will be chosen to match our purpose in the next proof. The important and well known facts are 
\begin{align} \label{gamma*togamma}
\|N_*(u)\|_{p}\lesssim \|N(u)\|_{p}, \qquad \|\wt N_*(u)\|_{p}\lesssim \|\wt N(u)\|_{p}, \quad \text{ and } \quad \| S_*(u)\|_{p}\lesssim \| S(u)\|_{p},
\end{align}
for every $p\in (0,\infty)$ and every $u$ for which the considered quantities make sense (the constants depends on $p$ but not $u$). The proof of \eqref{gamma*togamma} - in the case $\partial \Omega = \R^{n}$ - can be found in Chapter II, equation (25) of \cite{stein2016harmonic} for the non-tangential square function and in Proposition 4 of \cite{coifman1985some} for the square function. Although the proof is written when
$\partial\Omega=\mathbb{R}^n$, it can be easily extended to all doubling metric spaces.

\begin{Lemma}\label{LM25}
For the function $v$ constructed in \eqref{defv}, define the set
\begin{equation}
E_{\beta\alpha}:=\{N_*(v)+\wt N_*(\delta|\nabla v|)+\mathcal M_{\omega}(T(\vec h))>\beta\alpha\} \subset \partial \Omega.
\end{equation}
There exist $\eta, \beta_0>0$ such that for all $\alpha >0$ and $\beta \in (0, \beta_0)$, we have
\begin{align}\label{LM25.eq00}
    \sigma\{x\in \partial \Omega, \, S(v)(x)>2\alpha \text{ and } \mathcal M_\sigma (\1_{E_{\beta\alpha}})(x)\leq 1/2\}\leq C\beta^\eta\sigma\{S_*(v)>\alpha\},
\end{align}
where $C$ is independent of $\alpha$ and $\beta$.
\end{Lemma}

The above ``good-$\lambda$'' argument entails the following $L^p$ bounds. 

\begin{Cor}\label{CO26}
For any $p>0$, we have
\begin{align*}
    \int_{\partial \Omega} |S(v)|^{p}\, d\sigma\leq C\int_{\partial \Omega} |N(v)+\wt N(\delta|\nabla v|) + \mathcal M_{\omega}(T(\vec h))|^{p}\, d\sigma.
\end{align*} 
\end{Cor}

\noindent {\em Proof of Corollary \ref{CO26}:}
Let $E_{\beta\alpha}$ be the set defined in Lemma \ref{LM25}.
Recall that the Hardy Littlewood maximal operator $\mathcal M_{\sigma}$ is bounded from $L^1$ to weak-$L^1$. Then
\begin{align}\label{CO26.eq00}
    \sigma\{\mathcal M_\sigma (\1_{E_{\beta\alpha}})>1/2\}\lesssim \int_{\partial \Omega} \1_{E_{\beta\alpha}} d\sigma  = \sigma(E_{\beta\alpha}).
\end{align}
According to Lemma \ref{LM25}, we have, for $\beta \leq \beta_0$, that
\begin{equation}\label{CO26.eq02} \begin{split}
    \sigma\{S(v)>2\alpha\} & \leq \sigma\{\mathcal M_\sigma (\1_{E_{\beta\alpha}})>1/2\} +  \sigma\{S(v)>2\alpha, \mathcal M_\sigma (\1_{E_{\beta\alpha}})\leq 1/2\} \\
    & \leq  \sigma\{\mathcal M_\sigma (\1_{E_{\beta\alpha}})>1/2\} + C\beta^\eta\sigma\{S_*(v)>\alpha\}.
\end{split}\end{equation}
The last two computations imply, for any $p>0$, that
\begin{multline}\label{CO26.eq03}
    \int_{\partial \Omega} |S(v)|^{p}\, d\sigma= c \int_0^\infty \alpha^{p-1}\sigma\{S(v)>2\alpha\}d\alpha\\
    \leq C\beta^\eta  \int_0^\infty\alpha^{p-1}\sigma\{S_*(v)>\alpha\}\, d\alpha
    +C\int_0^\infty \alpha^{p-1}\sigma(E_{\beta\alpha})\, d\alpha\\
    \leq C\beta^\eta \int_{\partial \Omega} |S_*(v)|^{p}\, d\sigma+\frac{C}{\beta^{p}}\Big \{\int_{\partial \Omega} |\wt N_*(v) + \wt N_*(\delta|\nabla v|) + \mathcal M_{\omega}(T_{p}(\vec h))|^{p}\, d\sigma\Big \} \\
    \leq C' \beta^\eta \int_{\partial \Omega} |S(v)|^{p}\, d\sigma+\frac{C'}{\beta^{p}}\Big \{\int_{\partial \Omega} |\wt N(v) + \wt N(\delta|\nabla v|) + \mathcal M_{\omega}(T_{p}(\vec h))|^{p}\, d\sigma\Big \}
\end{multline}
by \eqref{gamma*togamma}. Choose a $\beta\leq \beta_0$ small enough so that $C'\beta^\eta \leq 1/2$. Hence, we can hide the square function of the last inequality of (\ref{CO26.eq03}) to the left-hand side. The corollary follows.
\ep

\bigskip

\noindent {\em Proof of Lemma \ref{LM25}:}
Fix $\alpha >0$. Define $\mathcal{S}:=\{S_*(v)>\alpha\}$ and $\mathcal{S}':=\{S(v)>2\alpha, \mathcal M_\sigma (\1_{E_{\beta\alpha}})\leq 1/2\}$. Take any surface ball $\Delta$ of radius $r$. It suffices to show that there exists a constant $C$ such that for any surface ball $\Delta$ that intersects $\partial \Omega \setminus \mathcal S$, we have
\begin{align}\label{LM25.eq01}
    \omega^\infty_{1,\infty}(F)\leq C\beta^2\omega^\infty_{1,*}(\Delta),
\end{align}
where $F:=\mathcal{S}' \cap \Delta$.  Indeed, the bound \eqref{RHom2} - which comes from the $L^{q'}$-solvability of the Dirichlet problem for $\mathcal L_1^*$ - immediately turns \eqref{LM25.eq01} into 
\begin{align}\label{LM25.eq30}
\sigma(F)\leq C\beta^\eta\sigma(\Delta).
\end{align}
Why is \eqref{LM25.eq30} enough? Because we can construct a Whitney decomposition of $\mathcal{S}$ in the following way. 
For any $x\in \mathcal{S}$, we can build the boundary ball $\Delta_x := \Delta(x,\dist(x,\mathcal{S}^c)/40)$.
Since the radius of a ball $\Delta_x$ that intersects a compact subset of $\mathcal S$ is uniformly bounded (depending on the compact), the Vitali covering lemma allows us to obtain a non-overlapping sub-collection $\{\Delta_{x_j}\}$ of $\{\Delta_x\}$ for which $\bigcup_j 5\Delta_{x_j} = \mathcal{S}$. Each ball $50\Delta_{x_j}$ intersects $\partial \Omega \setminus \mathcal S$ and so, if $F_j = \mathcal S' \cap 50\Delta_{x_j}$, we have by \eqref{LM25.eq30}
\[\sigma(\mathcal S') \leq \sum_j \sigma(F_j) \leq C\beta^\eta\sigma(50\Delta_j) \leq C' \beta^\eta\sigma(\Delta_j) \leq C' \beta^\eta \sigma(\mathcal S),\]
which is the desired bound \eqref{LM25.eq00}.

\medskip

\textbf{Step 1:} Let $\Delta$ be a surface ball of radius $r$ that contains a point $x_\Delta \in \partial \Omega \setminus \mathcal S$, i.e. a point satisfying $S_*(v)(x_\Delta)\leq \alpha$. We write $F:=\Delta\cap \mathcal{S'}$. 

Observe that for any $x\in \Delta$ and $X\in \gamma(x) \setminus B(x,r)$, we have
\[|X-x_\Delta| \leq |X-x| + |x-x_\Delta| < 2\delta(X) + 2r < 2\delta(X) + 2|X-x| < 6\delta(X).\]
Consequently, $\gamma(x) \setminus B(x,r) \subset \gamma^*(x)$ as long as the aperture $C_* \geq 6$ (which we choose as such). So if $S^{r}(v)(x)$ is a truncated square function defined as
\begin{align*}
    S^{r}(v)(x):=\Big (\int_{\gamma(x)\cap B(x,r)}|\nabla v|^2\frac{dY}{\delta(Y)^{n-2}}\Big )^{1/2},
\end{align*}
then we easily have  
\[|S^{r}(v)(x)|^2 \geq |S(v)(x)|^2 - |S_*(v)(x_\Delta)|^2 \geq \alpha^2,\quad \text{ for } x\in F,\]
 that is
\begin{align}\label{LM25.eq42}
S^{r}(v)(x) \geq \alpha,\quad \text{ for } x\in F.
\end{align}

\medskip

\textbf{Step 2:} In the sequel, to lighten the notation, we shall write $\omega$ for $\omega^\infty_{1,*}$, the elliptic measure with pole at infinity associated to $\mathcal L_1^*$. In a similar way, $G(Y)$ will denote the Green function with pole at infinity associated to $\mathcal L_1$. Both of them are linked together by Lemma \ref{CPWI01}.  Let $\Omega_F$ be the saw-tooth region over $F$ defined as $\Omega_F:=\bigcup_{x\in F}\gamma(x)$. Then
\begin{multline}\label{LM25.eq02}
    \omega(F)\leq \frac{1}{\alpha^2}\int_F |S^{r}(v)|^2 d\omega \leq \frac{1}{\alpha^2}\int_F\int_{\gamma(x)\cap B(x,r)}|\nabla v|^2\frac{dY}{\delta(Y)^{n-2}} d\omega(x)\\
    \leq \frac{1}{\alpha^2}\int_{\Omega_F\cap \{\delta(Y)\leq r\}}  |\nabla v|^2 \omega(B(Y,2\delta(Y)) \cap \partial \Omega) \frac{dY}{\delta(Y)^{n-2}}.
\end{multline}
If $y\in \partial \Omega$ is a point such that $|Y-y| = \delta(Y)$, then
\begin{equation} \label{omegaG}
\omega(B(Y,2\delta(Y)) \cap \partial \Omega) \approx \omega(\Delta(y,\delta(Y)) \approx \delta(Y)^{d-1} G(Y)
\end{equation}
by the doubling property of $\omega$ (Lemma \ref{DBPWI}) and then Lemma \ref{CPWI01}. We use the above estimate in \eqref{LM25.eq02} to obtain
\[\omega(F) \lesssim \frac1{\alpha^2} \int_{\Omega_F \cap \{\delta(Y)\leq r\}} |\nabla v|^2 G \, \frac{dY}{\delta(Y)^{n-d-1}}.\]
Let us recall that $dm(Y) = \delta(Y)^{d+1-n} dY$. Together with the ellipticity of matrix $\mathcal{A}_1^T$, we have
\begin{align}\label{LM25.eq32}
    \omega(F)\lesssim \frac{1}{\alpha^2}\int_{\Omega_F\cap \{\delta(Y)\leq r\}}\mathcal{A}^T_1\nabla v\cdot \nabla v\, G \, dm.
\end{align}

Choose a cut-off function $\phi_{F}\in C^\infty(\R^n)$ such that, $0 \leq \phi_F \leq 1$,  $\phi_{F}\equiv 1$ on $\Omega_F$, and $\phi_{F}$ is supported on a larger saw-tooth region $\Omega^3_F:=\bigcup_{x\in F} \gamma_3(x)$, with $\gamma_3(x) := \{X\in \Omega, \, |X-x| < 3\delta(X)\}$. In addition, we can always pick the cut-off function $\phi_F$ so that $|\nabla\phi_F(Y)|\lesssim 1/\delta(Y)$. Pick another smooth function $\phi_{r}$ such that $0\leq \phi_r \leq 1$, $\phi_{r}\equiv 1$ when $\delta(Y)\leq r$ and $\phi_{r}\equiv 0$ when $\delta(Y)\geq 2r$ and $|\nabla \phi_r| \leq 2/r$. Define $\Psi=\phi_F\phi_{r}$. Then we have
\begin{align}\label{LM25.eq43}
    |\nabla \Psi(Y)|\lesssim \frac{\1_{D_1}(Y)}{\delta(Y)}+\frac{\1_{D_2}(Y)}{\delta(Y)},
\end{align}
where $D_1:=\{Y\in \Omega^3_F\setminus \Omega_F, \delta(Y)\leq 2r\}$ and $D_2:=\{Y\in \Omega^3_F, r\leq \delta(Y)\leq 2r\}$. By the product rule, the term \eqref{LM25.eq32} can be rewritten as,
\begin{multline*}
    \omega(F)\lesssim \frac{1}{\alpha^2}\int_{\Omega}\mathcal{A}^T_1\nabla v\cdot \nabla v\ (G\Psi) \, dm\\
    =\frac{1}{\alpha^2}\Big (\int_{\Omega}\mathcal{A}^T_1\nabla v\cdot \nabla [vG\Psi] \, dm -\int_{\Omega}\mathcal{A}^T_1\nabla v\cdot \nabla G \, (v\Psi) \, dm
    -\int_{\Omega}\mathcal{A}^T_1\nabla v\cdot \nabla \Psi \, (Gv) \, dm\Big ) \\
     =:\frac1{\alpha^2}(I+II+III).
\end{multline*}
The lemma will be proven once we show that $I$, $II$, $III$ are all bounded by $C(\alpha \beta)^2 \omega(\Delta)$.

\medskip

\textbf{Step 3: The term $\mathbf I$.} We want to use Lemma \ref{lemv}, and so we need to check that $vG\Psi \in W_0$. We have that $v\in W_0$ (also by Lemma \ref{lemv}). Thanks to the elliptic theory recalled in Section \ref{Selliptic}, we also have that $v$ is H\"older continuous close to the boundary (when we are outside the support of $\vec h$) and $G\Psi \in W_0 \cap L^\infty(\Omega)$. 
So in order to get that $vG\Psi \in W_0$, we only need to explain why $v\in L^\infty_{loc}(\Omega)$. The control of solutions for inhomegeneous Dirichlet problem was not done in \cite{david2017elliptic}, but that is fine, because we only require local boundedness inside the domain, so we can use the result from the classical (unweighted) theory, which can be found in Theorem 8.17 of \cite{gilbarg2015elliptic}. 
Now, we apply Lemma \ref{lemv}, which entails that
\begin{multline*}
    I= - \int_{\Omega}\vec h \cdot \nabla(vG\Psi) \,  dY
    =-\int_{\Omega}\vec h\cdot \nabla v \, (G \Psi)\,  dY\\
    -\int_\Omega\vec h\cdot \nabla G \, (v\Psi) \,  dY -\int_\Omega\vec h\cdot \nabla \Psi \, (v G) \, dY:=I_1+I_2+I_3.
\end{multline*}

By the definition of $F \subset \mathcal S'$, for any $x\in F$, we have $\mathcal M_\sigma(\1_{E_{\beta\alpha}})(x)\leq 1/2$. Thus, for any surface ball $\Delta'\subset {\partial \Omega}$ which contains such a point $x \in F$, we necessary have $\sigma(\Delta'\cap E_{\beta\alpha})/\sigma(\Delta')\leq 1/2$. This implies that $\sigma(\Delta'\cap E^c_{\beta\alpha})/\sigma(\Delta')> 1/2$. The $A_\infty$-absolute continuity \eqref{RHom2} yields then
\begin{align}\label{LM25.eq04}
    \omega(\Delta'\cap E^c_{\beta\alpha})/\omega_*(\Delta')\geq c >0.
\end{align}
The comparison \eqref{omegaG} now entails that
\begin{align}\label{LM25.eq06}
    |I_1| \leq \int_{\Omega^3_F}|\vec h| |\nabla v| G \, dY\lesssim \int_{\Omega^3_F}|\vec h| |\nabla v| \, \omega(\Delta(y,\delta(Y))) \delta(Y)^{1-d} \, dY,
\end{align}
where we recall that $y$ is a point on the boundary such that $|Y-y| = \delta(Y)$.
Since $\Omega^3_{F}$ is a sawtooth region over $F$, there exists a constant $C_0$ ($C_0= 4$) such that for all $Y\in \Omega^3_{F}$,
\begin{align*}
F\cap \Delta(y, C_0\delta(Y))\neq \emptyset.
\end{align*}
Thus, by (\ref{LM25.eq04}), 
\begin{align}\label{LM25.eq05}
    \omega(\Delta(y, \delta(Y)))\leq \omega(\Delta(y, C_0\delta(Y)))\lesssim \omega(E^c_{\beta\alpha}\cap\Delta(y, C_0\delta(Y))).
\end{align}
Together with (\ref{LM25.eq05}), (\ref{LM25.eq06}) becomes
\begin{multline}\label{LM25.eq07}
    |I_1|\lesssim \int_{\Omega^3_F}|\vec h||\nabla v|\frac{\omega(E^c_{\beta\alpha}\cap\Delta(y, C_0\delta(Y)))}{\delta(Y)^{d-1}}dY
   \\
    \lesssim\int_{E^c_{\beta\alpha}\cap C'_0\Delta}\Big (\int_{\gamma_3(x)}\frac{|\nabla v||\vec h|}{\delta(Y)^{d-1}}dY\Big )d\omega(x).
\end{multline}
Recall that $\gamma_3(x) \subset \gamma_d(x)$, which is used in the construction of $T(\vec h)$ in \eqref{DEFTP}. By H\"older's inequality, 
\begin{multline} \label{LM25.I1}
    \int_{\gamma_3(x)}\frac{|\nabla v||\vec h|}{\delta(Y)^{d-1}}dY
    \lesssim \sum_{I \in \mathcal W_x} \Big (\fint_{I}\delta|\nabla v||\vec h|dY\Big )\ell(I)^{n-d} \\
    \lesssim \sum_{I \in \mathcal W_x}  \Big (\fint_{I}\delta^2|\nabla v|^2dY\Big )^{1/2}\Big (\fint_{I}|\vec h|^2 dY\Big )^{1/2} \ell(I)^{n-d} \lesssim  \wt N_*(\delta|\nabla v|)(x)T(\vec h)(x)
\end{multline}
if we choose the aperture $C_*$ of the cone $\gamma_*(x)$ that defines $\wt N_*$ big enough so that $I \subset \gamma^*(x)$ for all $I\in \mathcal W_x$. When $x\in E^c_{\beta\alpha}$, we have $\wt N_*(\delta|\nabla v|)(x)\leq \beta\alpha$. Therefore, if $x_0$ is any point in $E_{\beta\alpha}^c \cap C'_0\Delta$ (if the set is empty, then $I_1=0$ and there is nothing to prove), then  (\ref{LM25.eq07}) can be further continued as,
\begin{align}\label{LM25.eq99}
    |I_1|\lesssim (\beta\alpha)\int_{E^c_{\beta\alpha}\cap C_0\Delta}T(\vec h) \, d\omega \lesssim (\beta\alpha)\omega(C_0\Delta) \mathcal M_{\omega}(T(\vec h))(x_0)\lesssim (\beta\alpha)^2\omega(\Delta),
\end{align}
thanks to the doubling property of $\omega$ (Lemma \ref{DBPWI}) and the fact that $\mathcal M_{\omega}(T_p(\vec h))(x_0)\leq \beta\alpha$ for $x_0 \in E^c_{\beta\alpha}$. 

The term $I_3$ is very similar to $I_1$. Indeed, in $I_1$, we only use the fact that $0 \leq \Psi \leq 1$ and is supported in $\Omega^*_F$. For $I_3$, we use the fact that $|\nabla \Psi| \lesssim 1/\delta$ and is supported in $\Omega^*_F$, and we use $N_*(v)$ instead of $\wt N_*(\delta |\nabla v|)$. So with the same reasoning as of $I_1$, we also have
\begin{equation} \label{LM25.I3}
|I_3| \lesssim (\beta \alpha)^2 \omega(\Delta).
\end{equation}

The term $I_2$ is slightly more different from $I_1$ than $I_3$ is, so we shall write a bit more. Observe that $I_2$ is the same as $I_1$ once you replace $\nabla v$ by $v\frac{\nabla G}G$. So similarly to \eqref{LM25.eq07}, we have that
\begin{equation}\label{LM25.eq07b}
    |I_2|\lesssim \int_{E^c_{\beta\alpha}\cap C'_0\Delta}\Big (\int_{\gamma_3(x)}\frac{v|\nabla G||\vec h|}{G\delta(Y)^{d-1}}dY\Big )d\omega(x).
\end{equation}
Then analogously to \eqref{LM25.I1}, we get that
\begin{multline} \label{LM25.I1b}
    \int_{\gamma_3(x)}\frac{v|\nabla G||\vec h|}{G\delta(Y)^{d-1}}dY
    \lesssim \sum_{I \in \mathcal W_x} \Big (\fint_{I} v \frac{\delta |\nabla G|}{G}|\vec h|dY\Big )\ell(I)^{n-d} \\
    \lesssim \sum_{I \in \mathcal W_x}  \Big (\fint_{I} v^2 \, dY\Big )^{1/2} \Big (\fint_{I}\frac{\delta^2|\nabla G|^2}{G^2} dY\Big )^{1/2}  \sup_I |\vec h| \, \ell(I)^{n-d}.
\end{multline}
Since $G$ is a positive solution to $\mathcal L^1$, the Harnack inequality (Lemma \ref{LMHANK}) and the Cacciopoli inequality (Lemma \ref{LMICAE}) entail that
\begin{equation} \label{CaccioG}
\begin{split}
\fint_{I}\frac{\delta^2|\nabla G|^2}{G^2} dY \approx \frac{\delta(X_I)^2}{G(X_I)^2} \fint_{I} |\nabla G|^2 dm \lesssim \frac{1}{G(X_I)^2} \fint_{2I} G^2 dm \approx 1,
\end{split}\end{equation}
whenever $X_I$ is any point in $I$.
So the bound \eqref{LM25.I1b} becomes 
\[ \int_{\gamma_3(x)}\frac{v|\nabla G||\vec h|}{G\delta(Y)^{d-1}}dY \lesssim \sum_{I \in \mathcal W_x}  \Big (\fint_{I} v^2 \, dY\Big )^{1/2} \ell(I)^{n-d} \sup_I |\vec h| \lesssim  N_*(v)(x) T(\vec h)(x).\]
We use this last estimate in \eqref{LM25.eq07b} and we conclude that 
\begin{equation} \label{LM25.I2}
|I_2| \lesssim (\beta \alpha)^2 \omega(\Delta).
\end{equation}

\textbf{Step 4: Carleson estimates for $\mathbf{|\nabla \Psi|}$.} As we shall see, the terms $II$ and $III$ will only involve $\Psi$ via its gradient. So it will be useful to have good estimates on $\delta |\nabla \Psi|$, or on $\1_{D_1 \cup D_2}$ (which bigger by \eqref{LM25.eq43}). We aim to prove that
\begin{equation} \label{NablaPsiCM}
M_{\Delta} := \int_{C'_0\Delta} \left(\sum_{\begin{subarray}{c} I\in \mathcal W_x \end{subarray}} \sup_I (\1_{D_1 \cup D_2}) \right) \, d\omega(x) \leq C \omega(\Delta),
\end{equation}
where $C'_0$ is the constant on the right-hand side of \eqref{LM25.eq07}.

Even if the inequality \eqref{NablaPsiCM} is presented in an unusual way, the result is fairly classical. Let us sketch it. By simply switching the integral and the sum, we have
\[M_{\Delta} \leq \sum_{\begin{subarray}{c} I\in \mathcal W \\ I \cap  D_1 \neq \emptyset \end{subarray}} \omega(\Delta_I) + \sum_{\begin{subarray}{c} I\in \mathcal W \\ I \cap D_2 \neq \emptyset \end{subarray}} \omega(\Delta_I) := M_1 + M_2 ,\]
where $\Delta_I := \Delta(\xi_I,200\ell(I))$ for a point $\xi_I \in 100I \cap \partial \Omega$ that will be chosen later. If $I$ intersects $D_2$, then $\ell(I) \approx \dist(I,\Delta) \approx r$: there is a uniformly bounded amount of those cubes, and we also have $\omega(\Delta_I) \approx \omega(\Delta)$. We deduce 
\[M_2:= \sum_{\begin{subarray}{c} I\in \mathcal W \\ I \cap D_2 \neq \emptyset \end{subarray}} \omega(\Delta_I) \lesssim \omega(\Delta)\]
as desired. As for $J_1$, we use the fact that we have some freedom on the choice of $\xi_I$. If $I \cap D_1 \supset \{X_I\} \neq \emptyset$, then we choose $\xi_I \in \partial \Omega$ such that $|\xi_I - X_I| = \delta(X_I):= r_I$. Note that we necessary have $r_I \leq 60\ell(I)$, so $\xi_I \in 100I$. Recall that $X_I \in D_1$ means that there exists $x_I \in F$ such that $|X_I - x_I| < 3\delta(X_I) = 3|X_I - \xi_I|$ but $|X_I - x| \geq 2|X_I - \xi_I|$ for all $x\in F$. Consequently,
\begin{equation} \label{prxiI}
10\ell(I) \leq r_I \leq \dist(\xi_I, F) \leq 4 r_I.
\end{equation}
When $I \cap D_1 \neq \emptyset$, we define $\wt \Delta_I = \Delta(\xi_I, \ell(I))$ using the $\xi_I$ that we constructed and satisfies \eqref{prxiI}. Notice that the collection $\{\wt \Delta_I\}$ is finitely overlapping, because if $x\in \wt \Delta_I$, then $\ell(I) \approx \dist(x,F)$ and $I \subset B(x,62\ell(I))$, and there can be only a uniformly finite Whitney cubes with this property. We conclude by writing
\[M_1 := \sum_{\begin{subarray}{c} I\in \mathcal W \\ I \cap  D_1 \neq \emptyset \end{subarray}} \omega(\Delta_I)  \lesssim \sum_{\begin{subarray}{c} I\in \mathcal W \\ I  \cap  D_1 \neq \emptyset \end{subarray}} \omega(\wt \Delta_I) \lesssim \omega \left( \bigcup_{\begin{subarray}{c} I\in \mathcal W \\ I \cap  D_1 \neq \emptyset \end{subarray}}  \wt \Delta_I \right) \leq \omega(C''_0 \Delta) \lesssim \omega(\Delta).\]
The first and last inequalities above hold because of the doubling property of $\omega$ (Lemma \ref{DBPWI}), the second inequality is due to the finite overlap of $\{\wt \Delta_I\}$, and the third one is a consequence of the fact that all $\Delta_I$ (and thus $\wt \Delta$) are included in a dilatation of $\Delta$ when $I$ intersects $D_1$.

\medskip

\textbf{Step 5: The terms II and III.} Let us talk about $III$ first. We can repeat the strategy developed in Step 3 for $I_1$. We use the fact that $\delta \nabla \Psi \leq \1_{D_1 \cup D_2}$ and $dm(Y) = \delta(Y)^{d+1-n} dY$, and similarly to \eqref{LM25.eq07}, we have
\begin{equation}\label{LM25.eq07c}
    |III|   \lesssim\int_{E^c_{\beta\alpha}\cap C'_0\Delta}\Big (\int_{\gamma_3(x)}\frac{|\nabla v|v \1_{D_1 \cup D_2}}{\delta(Y)^{n-1}}dY\Big )d\omega(x).
\end{equation}
Yet, 
\[\begin{split}
\int_{\gamma_3(x)}\frac{|\nabla v|v \1_{D_1 \cup D_2}}{\delta(Y)^{n-1}}dY & \leq \sum_{I \in \mathcal W_x} \int_I \frac{|\nabla v|v \1_{D_1 \cup D_2}}{\delta(Y)^{n-1}}dY \\
& \lesssim \sum_{I \in \mathcal W_x}  \left(\fint_I \delta^2 |\nabla v|^2 \, dY \right)^\frac12 \left(\fint_I v^2 \, dY\right)^\frac12 \sup_I(\1_{D_1\cup D_2}) \\
& \lesssim \wt N_*(\delta |\nabla v|)(x) N_*(v)(x) \sum_{I \in \mathcal W_x} \sup_I(\1_{D_1\cup D_2}) \\
& \leq (\alpha\beta)^2 \sum_{I \in \mathcal W_x} \sup_I(\1_{D_1\cup D_2})
\end{split}\]
when $x\in E_{\beta\alpha}^c$. We conclude that
\[|III| \lesssim (\alpha\beta)^2 \int_{C'_0\Delta} \left(\sum_{I \in \mathcal W_x} \sup_I(\1_{D_1\cup D_2}) \right) d\omega(x) \lesssim (\alpha\beta)^2 \omega(\Delta) \]
by \eqref{NablaPsiCM} and the doubling property of $\omega$ (Lemma \ref{DBPWI}).

\medskip

For $II$, we want to use the fact that $G$ is a solution to $\mathcal L_1$, so we write
\begin{multline*}
   II=- \int_\Omega \nabla v \cdot \mathcal{A}_1\nabla G (v\Psi) \,  dm \\
    =-  \frac12 \int_\Omega \nabla (v^2\Psi) \cdot \mathcal{A}_1\nabla G \, dm + \int_\Omega v^2 \mathcal{A}_1\nabla G \cdot \nabla \Psi \, dm
    =:II_1+II_2.
\end{multline*}
The discussion at the beginning of Step 3 shows that $v \in W_0 \cap L^\infty(\supp \Psi)$. So $v^2 \Psi$ lies in $W_0$ and it is compactly supported in $\R^n$. Consequently, $v^2 \Psi$ is a valid test function for $G$, and thus $II_1 = 0$.
Hence it remains to bound $II_2$, which is actually similar to $III$. Following again the same strategy, replacing $|\nabla v|$ by $v|\nabla G|/G$ in the argument of $III$, we have
\[    |II_2|   \lesssim\int_{E^c_{\beta\alpha}\cap C'_0\Delta}\Big (\int_{\gamma_3(x)}\frac{v^2|\nabla G| \1_{D_1 \cup D_2}}{G \delta(Y)^{n-1}}dY\Big )d\omega(x)\]
and when $x\in E_{\alpha\beta}^c$
\[\begin{split}
\int_{\gamma_3(x)}\frac{v^2|\nabla G| \1_{D_1 \cup D_2}}{G \delta(Y)^{n-1}}dY 
& \lesssim \sum_{I \in \mathcal W_x}  \left(\fint_I \frac{\delta^2|\nabla G|^2}{G^2}  \, dY \right)^\frac12 \left(\fint_I v^4 dY\right)^\frac12 \sup_I(\1_{D_1\cup D_2}) \\
& \lesssim |N_*(v)(x)|^2 \sum_{I \in \mathcal W_x} \sup_I(\1_{D_1\cup D_2}) \\
& \leq (\alpha\beta)^2 \sum_{I \in \mathcal W_x} \sup_I(\1_{D_1\cup D_2}),
\end{split}\]
where we used \eqref{CaccioG} for the second inequality. With a similar reasoning as the one used on $III$, we conclude that 
\[|II| = |II_2| \lesssim (\alpha\beta)^2 \omega(\Delta)\]
thanks to \eqref{NablaPsiCM}.  The lemma follows.
\ep

\subsection{Bounds on $N(v)$ and $\tilde N(\delta \nabla v)$}

In order to finish the proof Theorem \ref{thm:mn1}, we need to bound $N(v)$ and $\wt N(\delta \nabla v)$ by $T(\vec h)$. We shall observe first that the bound on $\wt N(\delta \nabla v)$ is just a consequence of the bound on $N(v)$ because of the following Caccioppoli-type inequality.

\begin{Lemma}\label{LMCPVE}
For any $X\in \Omega$, we have
\begin{equation}\label{LM45.eq05}
\Big (\fint_{B_X} \delta^2|\nabla v|^2 \, dX \Big )^{1/2}
\lesssim \Big (\int_{2B_X}|v|^2\, dX\Big )^{1/2} + \delta(X)^{n-d} \Big (\fint_{2B_X}|\vec h|^2  \, dY \Big )^{1/2},
\end{equation}
where $v$ is constructed in \eqref{defv}. 
\end{Lemma}

\begin{proof}
Take $X\in \Omega$ and construct a cut-off function $\Psi\in C_0^\infty(\Omega)$ such that $0 \leq \Psi \leq 1$, $\Psi\equiv1$ on $B_X$, $\Psi\equiv 0$ outside $2B_X$, and $|\nabla \Psi|\lesssim 1/\delta(X)$. By the ellipticity of $\mathcal{A}^T$, we have
\begin{multline}\label{LM23.eq95}
     T:=\int_\Omega |\nabla v|^2\Psi^2dm
    \lesssim \int_\Omega \mathcal{A}^{T}_1\nabla v\cdot \nabla v\Psi^2dm\\
    =\int_\Omega \mathcal{A}^{T}_1\nabla v\cdot \nabla [v\Psi^2] dm-\int_\Omega \mathcal{A}^{T}_1\nabla v\cdot \nabla\Psi \, \Psi v\, dm=:T_1+T_2.
\end{multline}
We want to use the fact that $v$ is a solution to $\mathcal L_1^* v = \diver \vec h$. Observe that $v\Psi^2$ lies in $W_0$, hence it is a valid test function, because $v\in W_0$ lies in $L^\infty_{loc}$ (we refer to the discussion at the beginning of Step 3 of the proof of Lemma \ref{LM25}) and $\Psi\in C_0^\infty$. Lemma \ref{lemv} entails
\begin{multline*}
    T_1= - \int_\Omega \vec h \cdot \nabla(v\Psi^2) \, dY=-\int_\Omega \vec h\cdot \nabla v\, \Psi^2 \, dY-\int_\Omega \vec h\cdot \nabla \Psi \,  \Psi v \, dY\\
    \lesssim  \Big (\delta(X)^{n-d-1} \int_\Omega |\vec h|^2\Psi^2 dY \Big )^{1/2} \left[ \Big (\int_\Omega |\nabla v|^2\Psi^2\, dm\Big )^{1/2}+ \Big (\int_\Omega | v|^2|\nabla\Psi|^2dm\Big )^{1/2} \right]
\end{multline*}
by the Cauchy-Schwarz inequality and $dm(Y) \approx \delta(X)^{d+1-n} dY$ when $Y\in \supp \Psi \subset 2B_X$. We use the fact that $\Psi$ is supported on $2B_X$ and $|\nabla \Psi|\lesssim 1/\delta(X)$ to further have
\begin{align}\label{LM23.eq90}
    |T_1|\lesssim \Big ( \delta(X)^{n-d-1} \int_{2B_X} |\vec h|^2  dY \Big )^{1/2}  \left[  T^{1/2} +  \delta(X)^{-1} \Big (\int_{2B_X} | v|^2dm\Big )^{1/2} \right]
\end{align}
Similarly, $T_2$ is bounded using the Cauchy-Schwarz inequality and the properties of $\Psi$ by
\begin{align}\label{LM23.eq91}
    |T_2|\lesssim T^{1/2} \delta(X)^{-1} \Big (\int_{2B_X} |v|^2\, dm\Big)^{1/2}. 
\end{align}

Finally, by applying the estimates (\ref{LM23.eq90}) and (\ref{LM23.eq91}) to (\ref{LM23.eq95}), we deduce that $T \lesssim A^{1/2} T^{1/2} + A$ where 
\[A := \delta(X)^{-1} \int_{2B_X}|v|^2\, dX + \delta(X)^{n-d-1} \fint_{2B_X}|\vec h|^2  \, dY.\]
Since all the quantities that we considered are finite, this bound on $T$ self improves to $T\lesssim A$. The lemma follows easily.
\end{proof}

\begin{Lemma}\label{LM23}
Let $v$ be the weak solution constructed in \eqref{defv}.
\begin{align}\label{LM23.eq1}
    N(v)+\wt N(\delta|\nabla v|)\leq C\mathcal M_{\omega}(T(\vec h)),
\end{align}
where $\mathcal M_{\omega}$ is the Hardy–Littlewood maximal function with respect to $\omega:= \omega^\infty_{1,*}$, the elliptic measure with pole at infinity associated to $\mathcal L_1^*$. 
\end{Lemma}

\begin{proof}
Fix $x_0\in \partial \Omega$ and then $X\in \gamma_*(x_0)$, where $\gamma_*(x)$ is a cone with a bigger aperture so that $\bigcup_{Y\in \gamma(x)} 2B_Y \subset \gamma_*(x)$. We want to show that
\begin{align}\label{LM23.eq93}
  |\vec h(X)|\delta(X)^{n-d}\lesssim \mathcal M_{\omega}(T(\vec h))(x_0).
\end{align}
and
\begin{align}\label{LM23.eq15}
    v(X)\lesssim \mathcal M_{\omega}(T(\vec h))(x_0).
\end{align}
Indeed, once these two estimates are proven, then the bound  $\wt N(\delta|\nabla v|)\lesssim \mathcal M_{\omega}(T(\vec h))$ will follow thanks to Lemma \ref{LMCPVE}. 
The bound \eqref{LM23.eq93} is also fairly immediate. Take $x$ such that $|X-x| = \delta(X)$, and check that $X \in \gamma(y)$ for any $y$ in a small boundary ball $\Delta(x,c\delta(X))$. Hence, we easily have $ |\vec h(X)|\delta(X)^{n-d} \leq T(\vec h)(y)$ for $y\in \Delta(x,c\delta(X))$ by definition of $T(\vec h)$, and $\Delta(x,c\delta(X)) \subset \Delta(x_0,C\delta(X))$ for $C$ large enough depending only on the aperture of $\gamma_*(x_0)$. The inequality \eqref{LM23.eq93} follows.

\medskip

It remains to show \eqref{LM23.eq15}. By definition, 
\begin{align*}
    v(X)   :=-\int_{\Omega}\nabla_Y G(Y,X)\cdot \vec h(Y)dY,
\end{align*}
where $G(Y,X)$ is the Green function with pole at $X$ associated to $\mathcal{L}_1$. We shall treat differently the cases where $Y$ is close to $X$ and far from $X$. We define $S_X$ as the union of Whitney cubes $I\in \mathcal W$ (constructed in Subsection \ref{SSWhitney}) for which $3I \ni X$. The function $v(X)$ can be decomposed as
\begin{align}\label{LM23.eq20}
    v(X)=- \int_{\Omega\setminus S_X} \nabla_Y G(Y,X)\cdot \vec h(Y)\, dY -\int_{S_X}\nabla_Y G(Y,X)\cdot \vec h(Y)\, dY := \wt v + v_0. 
\end{align}

\medskip

\textbf{Step 1: Bound on $\mathbf{\wt v}$.} By definition of $S_X$, we have
\begin{equation} \label{bddwtv2} \begin{split}
|\wt v(X)| & = \left| \sum_{\begin{subarray}{c} I \in \mathcal W \\ X \notin 3I \end{subarray}} \int_I  \nabla_Y G(Y,X)\cdot \vec h(Y)dY  \right|\\
& \leq \sum_{\begin{subarray}{c} I \in \mathcal W \\ X \notin 3I \end{subarray}} \left( \int_I |\nabla_Y G(Y,X)|^2 dY \right)^\frac12 \left(\int_I |\vec h(Y)|^2 dY \right)^\frac12 \\
& \lesssim \sum_{\begin{subarray}{c} I \in \mathcal W \\ X \notin 3I \end{subarray}} \ell(I)^{-1} \left( \int_{2I} G(Y,X)^2 dY\right)^\frac12 \left( \int_I |\vec h(Y)|^2 dY \right)^\frac12
\end{split}\end{equation}
by H\"older's inequality, and then by Caccioppoli's inequality (see Lemma \ref{LMICAE}, that we can use because $G(.,X)$ is a solution on $2I$).

We want now to estimate $G(Y,X)$. Pick a point $Y_I$ in $I$. By the Harnack inequality (Lemma \ref{LMHANK}), we have $G(Y,X) \approx G(Y_I,X)$ for all $Y\in 2I$. So \eqref{bddwtv2} becomes
\begin{equation} \label{bddwtv}
|\wt v(X)| \lesssim \sum_{\begin{subarray}{c} I \in \mathcal W \\ X \notin 3I \end{subarray}} \ell(I)^{n-1} G(Y_I,X)  \sup_I |\vec h|.
\end{equation}
Our next objective is \eqref{GYXGY}. We give the details, but a reader who is an expert in the elliptic theory may want to skip them. First, we shall introduce several notations. For $j\geq 1$, let us denote by $\Delta_j:=\Delta(x_0,2^j\delta(X))$ the boundary balls and $X_j$ the Corkscrew points associated to $\Delta_j$. It is also fair to pick $X_1 := X$. We partition $\mathcal W$ into $\bigcup_{j\geq 1} \mathcal W_j$, where 
\[\mathcal W_1 := \{I \in \mathcal W, \, |Y_I - X| \leq 2\delta(X) \}\]
and for $j\geq 2$,
\[\mathcal W_j := \{I \in \mathcal W, \, 2^{j-1}\delta(X) < |Y_I - X| \leq 2^j\delta(X) \}.\]
Observe that we can find an integer $a$ that depends only on $n$ and the aperture of $\gamma_*(x_0)$ such that we have
\[2I \in B(x_0,2^{j+a}\delta(X)) \text{ for $I\in \mathcal W_j$, $j\geq 1$} \quad \text{ and } \quad 2I \cap B(x_0,2^{j-a}\delta(X)) = \emptyset \text{ for $I\in \mathcal W_j$, $j\geq 2$}.\]
So for each $j\geq 1$, we take $j_-$ to be the biggest value for which $X_{j_-}$ stays in $B(x_0,2^{j-a-1}\delta(X))$ and $j_- = 1$ if there are none, and we take $j_+$ to be the smallest value for which $X_{j_+}$ is outside  $B(x_0,2^{j+a+1})$. Note that by construction, $|j-j_-| + |j-j_+| \lesssim 1$. 
When $I\in \mathcal W_j$, the function  $G(Y_I,.)$ is a solution on $B(x_0,2^{j+a}\delta(X))$. Therefore,  the H\"older continuity at the boundary (Lemma \ref{LMHIB}) entails that
\[G(Y_I,X) \lesssim 2^{-j\alpha} G(Y_I,X_{j_-}).\]
Our choice of $X_{j_-}$ and $X_{j_+}$ allows the construction of a Harnack chain of balls of (uniformly) finite length that link $X_{j_-}$ to $X_{j_+}$ and avoid $B_{Y_I}$. So by the Harnack inequality (Lemma \ref{LMHANK}), the above estimate is equivalent to
\[G(Y_I,X) \lesssim 2^{-j\alpha} G(Y_I,X_{j_+}).\]
Lemma \ref{CPWI01} implies now that
\[G(Y_I,X_{j_+}) \approx \frac{G(Y_I)}{\omega(\Delta_j)},\]
where $G$ is the Green function with pole at infinity associated to $\mathcal L_1$. If $\Delta_I := \Delta(\xi_I,\ell(I))$ - with $\xi_I$ such that $|Y_I - \xi_I| = \delta(Y_I)$ - we have by Lemma \ref{CPWI01} that
\[G(Y_I) \approx \ell(I)^{1-d} \omega(\Delta_I).\]
Altogether, our discussion of $G(Y_I,X)$ proves that
\begin{equation} \label{GYXGY}
G(Y_I,X) \lesssim 2^{-j\alpha} \ell(I)^{1-d} \frac{\omega(\Delta_I)}{\omega(\Delta_j)} \qquad \text{ when } I \in \mathcal W_j.
\end{equation}

We inject our estimate \eqref{GYXGY} in \eqref{bddwtv} to obtain that
\[\begin{split}|\wt v(X)| & \lesssim \sum_{j\geq 1}\frac{2^{-j\alpha}}{\omega(\Delta_j)}  \sum_{\begin{subarray}{c} I \in \mathcal W_j \end{subarray}} \omega(\Delta_I) \ell(I)^{n-d}  \sup_I |\vec h|.
\end{split}\]
Since $x\in \Delta_I$ implies that $I \in \mathcal W_x$, by Fubini's theorem, we have that
\[\sum_{\begin{subarray}{c} I \in \mathcal W_j \end{subarray}} \omega(\Delta_I) \ell(I)^{n-d}  \sup_I |\vec h| \lesssim \int_{C\Delta_j} T(\vec h)(x) \, d\omega(x) \]
and thus, thanks to the doubling property of $\omega$ (Lemma \ref{DBPWI}),
\[|\wt v(X)| \lesssim  \sum_{j\geq 1} 2^{-j\alpha} \fint_{C\Delta_j} T(\vec h) \, d\omega \lesssim \mathcal M_\omega(T(\vec h))(x_0),\]
which is our desired bound on $\wt v$.

\medskip

\textbf{Step 2: Bound on $\mathbf{v_0}$.} It remains to bound the term $\int_{S_X}\nabla_Y G(Y,X)\cdot \vec h(Y)dY$ in (\ref{LM23.eq20}). 

Since $dY \approx \delta(X)^{n-d-1} dm(Y)$ on $S_X $, the bound (iv) of Lemma \ref{LMGNEX} shows that 
\[\int_{S_X} |\nabla_Y G(Y,X)| dY \lesssim \delta(X)^{n-d-1} \int_{S_X} |\nabla_Y G(Y,X)| dm(Y) \lesssim  \delta(X)^{n-d}.\]
Therefore, we have
\[|v_0(X)| \lesssim \sum_{\begin{subarray}{c} I \in \mathcal W \\ 3I \ni X \end{subarray}}\ell(I)^{n-d}  \sup_{I} |\vec h|.\]
For each $I \in \mathcal W$, we can pick any point $Y_I \in I$ as before and then $y_I$ such that $|Y_I - y_I| = \delta(Y_I)$. It is fairly easy to see that $I \in \mathcal W_x$ for all $x \in \Delta(y_I,c\ell(I))$, with $c$ small enough independent of $I$, and thus $\ell(I)^{n-d}  \sup_{I} |\vec h| \leq T(\vec h)(x)$ for $x\in \Delta(y_I,c\ell(I))$. It infers that
\[\ell(I)^{n-d}  \sup_{I} |\vec h| \leq \fint_{\Delta(y_I,c\ell(I))} T(\vec h) \, d\omega.\]
If $X\in 3I \cap \gamma_*(x_0)$, we necessary have $\Delta(y_I,c\ell(I)) \subset \Delta(x_0,C\ell(I))$ for $C$ large enough. By the doubling property of $\omega$ (Lemma \ref{DBPWI}), we obtain
\[\ell(I)^{n-d}  \sup_{I} |\vec h| \lesssim \fint_{\Delta(x_0,C\ell(I))} T(\vec h) \, d\omega \leq \mathcal M_\omega (T(\vec h))(x_0).\]
Since the number of Whitney cubes $I \in \mathcal W$ for which $3I \ni X$ is (uniformly) finite, we can conclude that
\[|v_0(X)| \lesssim \mathcal M_\omega (T(\vec h))(x_0)\]
as desired. The lemma follows. 
\end{proof}

\appendix

\section{The regularity problem  implies the Dirichlet problem}

This section is devoted to the proof of Theorem \ref{ThRq=>Dq'}. We shall follow closely the proof of Theorem 5.4 in \cite{kenig1993neumann}. 
Note that when the operator is the Laplacian and the domain does not have Harnack chains, this result was proved by Mourgoglou and Tolsa as Theorem 1.5 in \cite{mourgoglou2021regularity}. Since existing literature does not cover operators more general than the Laplacian, we decided to rewrite a proof in our context. 

In all this section, we assume that $\Omega$ is a uniform domain (see Definition \ref{defuniform}), that $\mathcal L = - \diver [w\A \nabla]$ is an elliptic operator satisfying \eqref{ELLIP}. 

\medskip

The following Poincar\'e inequality will needed.

\begin{Lemma}\label{dualm01}
For any $\alpha\in [0,1)$, any $x\in \partial \Omega$, any $r>0$, and function $u\in W(B(x,2r))$ satisfying $\Tr (u)\equiv 0$ on $\Delta(x,2r)$, we have
\begin{equation}\label{dualm01.eq01}
    \int_{B(x,r) \cap \Omega} |u(Y)|^2\delta(Y)^{\alpha} dm(Y)
    \leq {C_\alpha}{r^2}\int_{B(x,2r) \cap \Omega} |\nabla u(Y)|^2 \delta(Y)^{\alpha} dm(Y),
\end{equation}
where $C_\alpha$ depends on (the uniform constants of) $\Omega$ and $\alpha$.
\end{Lemma}

\begin{proof}
Define $dm'(X):= \delta(X)^\alpha dm(X)$ on $\Omega$, and then check that the triple $(\Omega,m',\sigma)$ satisfies the assumptions (H1) to (H6) from  \cite{david2020elliptic}. 
The result is then a consequence of Theorem 7.1 in \cite{david2020elliptic}. 
\end{proof}

\begin{Lemma}\label{dualm02}
Let $u\in W$ be a non-negative weak solution to $\mathcal{L}u=0$ that satisfies $\Tr (u)\equiv 0$ on $\Delta(x,r)$, then for each $X\in \Omega$ such that $|X-x| \approx \delta(X) \approx r$, we have
\begin{equation} \label{dualA} \frac{u(X)}{r} \approx \left( \fint_{B(x,r/2) \cap \Omega} |\nabla u(Y)|^2 dm(Y) \right)^\frac{1}{2} \lesssim \wt N_*(\nabla u)(x).
\end{equation}
Here $\wt N_*$ is defined with cones $\gamma^*(x) := \{X\in \Omega, \, |X-x| \leq C^* \delta(X)\}$ of large aperture. Besides, $C^*$ and the implicit constants in \eqref{dualA} depend only on the uniform constants of $\Omega$ and the constants in $|X-x| \approx \delta(X) \approx r$.
\end{Lemma}

\begin{proof}
{\bf Step 1:} We have that 
\begin{equation} \label{dualB} r^2 \fint_{B(x,r/2) \cap \Omega} |\nabla u(Y)|^2 dm(Y) \lesssim u(X)^2.
\end{equation}
Indeed, since $\Tr(u) = 0$ on $\Delta(x,\delta(X))$, the above bound is due to two basic results from \cite{david2020elliptic} - Lemma 11.15 (Caccioppoli's inequality at the boundary) and Lemma 15.14 - which, used in this order, gives that
\[ r^2 \fint_{B(x,r/2) \cap \Omega} |\nabla u(Y)|^2 \, dm \lesssim \fint_{B(x,3r/4) \cap \Omega} |u(Y)|^2 \, dm \lesssim |u(X)|^2.\]

\medskip

\noindent {\bf Step 2:} We claim that for any $\alpha \in [0,1)$, we have
\begin{equation} \label{dualC} u(X)^2 \lesssim r^{2-\alpha} \fint_{B(x,r/2) \cap \Omega} |\nabla u(Y)|^2 \delta(Y)^\alpha dm(Y).
\end{equation}
Let $X' \in \Omega \cap B(x,r/8)$ be such that $\delta(X') \approx r$, such point exists because $\Omega$ satisfies the corkscrew point condition (see Definition \ref{defCPC}). Thanks to the Harnack chain condition (Definition \ref{defHCC}) and the Harnack inequality (Lemma \ref{LMHANK}), we have $u(X) \approx u(Y)$ for any $Y \in B_{X'}$. So we obtain that 
\begin{multline*}
u(X)^2 \lesssim \fint_{B_{X'}} |u(Y)|^2 \, dm \approx r^{-\alpha} \fint_{B_{X'}} |u(Y)|^2 \delta(Y)^\alpha \, dm \\ \lesssim r^{-\alpha} \fint_{B(x,r/4) \cap \Omega} |u(Y)|^2 \delta(Y)^\alpha \, dm \\
\lesssim r^{2-\alpha} \fint_{B(x,r/2) \cap \Omega} |\nabla u(Y)|^2 \delta(Y)^\alpha dm(Y),
\end{multline*}
where we used the Poincar\'e inequality (Lemma \ref{dualm01}), and we can because $\Tr(u) = 0$ on $\Delta (x,r/2)$.

\medskip

\noindent {\bf Step 3: Conclusion.} The equivalence in \eqref{dualA} is the combination of \eqref{dualB} and \eqref{dualC} for $\alpha = 0$. It remains to prove the second bound in \eqref{dualA}, that is 
\begin{equation} \label{dualD}
\left( \fint_{B(x,r/2) \cap \Omega} |\nabla u(Y)|^2 dm(Y) \right)^\frac{1}{2} \lesssim \wt N_*(\nabla u)(x),
\end{equation}
but this bound is an immediate consequence of 
\begin{equation} \label{dualE}
 \int_{B(x,r/2) \cap \Omega} |\nabla u(Y)|^2 dm(Y)  \lesssim  \int_{B(x,r/2) \cap \Omega,} \1_{\delta(Y) > \epsilon_* r} |\nabla u(Y)|^2 dm(Y),
\end{equation}
where $\epsilon_*$ is a small constant that depends only on the uniform constants of $\Omega$, because the right-hand side of \eqref{dualE} is bounded by $|\wt N_*(\nabla u)(x)|^2$ if $C^*$ is large enough (depending on $\epsilon_*$).

In order to establish \eqref{dualE}, observe that \eqref{dualB} and \eqref{dualC} gives that
\[\int_{B(x,r/2) \cap \Omega} |\nabla u(Y)|^2 dm(Y) \lesssim r^{-\alpha} \int_{B(x,r/2) \cap \Omega} |\nabla u(Y)|^2 \delta(Y)^{\alpha} dm(Y)\]
and thus
\begin{multline*}
\int_{B(x,r/2) \cap \Omega} |\nabla u(Y)|^2 dm(Y) \leq C (\epsilon_*)^{\alpha} \int_{B(x,r/2) \cap \Omega} \1_{\delta(Y) \leq \epsilon_* r} |\nabla u(Y)|^2 dm(Y) \\ 
+ C \int_{B(x,r/2) \cap \Omega} \1_{\delta(Y) > \epsilon_* r} |\nabla u(Y)|^2 dm(Y).
\end{multline*}
We choose $\alpha=\frac12$ $\epsilon_* >0$ such that $C(\epsilon_*)^\alpha \leq \frac12$, so that we can hide the integral over $B(x,r/2) \cap \Omega \cap \{\delta(Y) \leq \epsilon_* r\}$ in the left hand side. The claim \eqref{dualE} and thus the lemma follow.
\end{proof}

We are know ready for the proof of Theorem \ref{ThRq=>Dq'}.

\begin{proof}[Proof of Theorem \ref{ThRq=>Dq'}]
Suppose that the regularity problem (defined using Haj\l asz-Sobolev spaces) for $\mathcal L$ is solvable in $L^q$. Let $\omega_{*}$ be the harmonic measure with pole at infinity associated to $\mathcal{L}^*$, that is defined in Definition \ref{DEFGI}. By Corollary \ref{RVHWI}, in order to show the Dirichlet problem for $\mathcal{L}^*$ is solvable in $L^{q'}$ , it suffices to show $\omega_{*}\ll \sigma$ and $k:=d\omega_{*}/d\sigma$ satisfies the reverse H\"older inequality of order $q$.

\medskip

\noindent \textbf{Step 1: } Thanks to the Ahlfors regularity  of $\partial \Omega$, for any boundary ball $\Delta$, there exists $K$ (that depends only on the constant $C_\sigma$ in \eqref{DEFSIG} such that $K\Delta \setminus 3\Delta \neq \emptyset$.

Let $\Delta := \Delta(x,r)$ be a surface ball on $\partial \Omega$. 
We construct $f$ on $\partial \Omega$ as
\begin{equation}
f(y) := \max\Big\{0, 1- \frac{\dist(y,K\Delta \setminus 3\Delta)}{r}\Big\}
\end{equation}
Note that $f$ is a non-negative function with $f\equiv 0$ on $2\Delta$ and $\partial \Omega\setminus (K+1)\Delta$ and $f\equiv 1$ on $K\Delta \setminus 3\Delta$. The function $f$ is Lipschitz, and if we define $g$ on $\partial \Omega$ as $g=\frac1r \1_{(K+1)\Delta}$, we easily have that 
\[|f(y) - f(z) | \leq |y-z| (g(y)+g(z)).\]
We deduce that $g$ is a generalized (or Haj\l asz upper) gradient of $f$, and thus the Haj\l asz Sobolev norm of $f$ satisfies
\begin{equation} \label{normf}
\|f\|_{\dot W^{1,q}} \leq C r^{d/q-1},
\end{equation}
where $C$ depends only on the Ahlfors regular constant $C_\sigma$.

Let $u$ be defined from $f$ as in \eqref{defhm2}, that is
\[u(X) := \int_{\partial \Omega} f(y) d\omega_{\mathcal L}^X(y).\]
Let $X_0 \in \Omega$ be a corkscrew point for $\Delta$, then 
\begin{equation} \label{duality.eq01}
u(X_0) \approx 1. 
\end{equation}
Indeed, the upper bound is 1 and comes from the fact that $\omega^X$ is a probability measure. The lower bound comes from the non-degeneracy of the harmonic measure: Since by definition of $K$, the set $K\Delta \setminus 3\Delta$ is non-empty , we can take $y\in K\Delta \setminus 3\Delta$, and then $Y_0$ a corkscrew point of $\Delta(y,r)$. The  
non-degeneracy of the harmonic measure (see for instance Lemma 15.1 in \cite{david2020elliptic}) gives that $u(Y_0)\gtrsim 1$, because $f$ is nonnegative and $f\geq 1/2$ on $\Delta(y,r/2)$. But $Y_0$ and $X_0$ can be linked by a Harnack chain, so the Harnack inequality (Lemma \ref{LMHANK}) entails that $u(X_0) \gtrsim 1$ as well.

\medskip

\noindent \textbf{Step 2:} In this step, we claim that for any $y\in \Delta(x,r)$, any $0<s<r/2$, and any $z\in \Delta(y,s)$
\begin{align}\label{duality.eq20}
    \frac{\omega_{*}(\Delta(y,s))}{\sigma(\Delta(y,s))} \approx \frac{\omega_{*}(\Delta(x,r))}{r^{d-1}} \wt N_*(\nabla u)(z). 
\end{align}

Let $G(.,.)$ and $G^\infty$ be the Green function and the Green function with pole at infinity respectively, in particular $G(.,Y)$ and $G^\infty$ are solution to $\mathcal L u=0$.
Both $G^\infty$ and $u(.)$ are non-negative solutions for which $\Tr(u) = \Tr(G^\infty) = 0$ on $2\Delta$, so by the comparison principle (Theorem \ref{THCPMLD})  and (\ref{duality.eq01}), we have, 
\begin{align}\label{duality.eq02}
    \frac{u(Y)}{G^\infty(Y)} \approx \frac{u(X_0)}{G^\infty(X_0)} \approx \frac{1}{G^\infty(X_0)} \qquad \text{ for $Y\in B(x,3r/2) \cap \Omega$}.
\end{align}
In addition, according to Lemma \ref{CPWI01},
\begin{align}\label{duality.eq03}
    G^\infty(X_0) \approx r^{1-d}\omega_{*}(\Delta(x,r)). 
\end{align}
Then combining (\ref{duality.eq02}), (\ref{duality.eq03}) and Lemma \ref{CPWI01} again, we obtain
\begin{align}\label{duality.eq04}
    \frac{u(Y)}{\delta(Y)}\approx \frac{G^\infty(Y)}{\delta(Y)} \frac{r^{d-1}}{\omega_{*}(\Delta(x,r))}\approx \frac{\omega_{*}(\Delta(y,s))}{s^d}\frac{r^{d-1}}{\omega_{*}(\Delta(x,r))},
\end{align}
where $s \approx \delta(Y)$ and $|Y-y| \lesssim s$. So if at the contrary we choose any $y\in \Delta$ and $0<s<r/2$, we take $Y \in B(y,s) \cap \Omega$ to be such that $\delta(Y) \gtrsim s$, \eqref{duality.eq04} and Lemma \ref{dualm02} entail
\begin{equation}\label{duality.eq05}
    \frac{\omega_{*}(\Delta(y,s))}{s^d}\approx \frac{\omega_{*}(\Delta(x,r))}{r^{d-1}}\frac{u(Y)}{\delta(Y)}\\
    \approx \frac{\omega_{*}(\Delta(x,r))}{r^{d-1}} N_*(\nabla u)(z).
\end{equation}
The claim (\ref{duality.eq20}) follows for the Ahlfors regularity of $\sigma$.

\medskip

\noindent \textbf{Step 3:} Assume that $E\subset \Delta$ and $\sigma(E) = 0$. Since $\omega_*$ is Borel regular, we have that $\omega_*(E) = \inf_{\begin{subarray}{c} V \supset E \\ V \text{ open} \end{subarray}} \omega_*(V)$. For each open set $V$, we cover it by the balls $\{B_y:= B(x,\dist(y,\partial \Omega \setminus V))\}_{y\in V}$ and using Vitali's covering lemma, we find a  sequence $\{y_i\}_{i\in V}$ such that $B_{y_i}$ are not overlapping while $5B_{y_i}$ covers $V$. By using \eqref{duality.eq20} on the balls $5B_{y_i}$, we find that 
\[\omega_*(5B_{y_i}) \approx \frac{\omega_{*}(\Delta(x,r))}{r^{d-1}} \int_{B_{y_i}} \wt N_*(\nabla u)(z) \, d\sigma(z),\]
and then
\[\omega_*(V) \approx \frac{\omega_{*}(\Delta(x,r))}{r^{d-1}} \int_{V} \wt N_*(\nabla u)(z) \, d\sigma(z).\]
Since we assume that the regularity problem is solvable in $L^q$, the function $\wt N(\nabla u)$ lies in $L^q(\partial \Omega,\sigma)$, and so by \eqref{gamma*togamma}, the function $N_*(\nabla u)$ lies in $L^q(\partial \Omega,\sigma)$. We invoke the Borel regularity of $\sigma$ to deduce that
\[\omega_*(E) = \inf_{\begin{subarray}{c} V \supset E \\ V \text{ open} \end{subarray}}  \int_V \wt N_*(\nabla u) \, d\sigma = \int_E \wt N_*(\nabla u) \, d\sigma =0.\]
We conclude that $\omega_* \ll \sigma$.

\medskip

\noindent \textbf{Step 4:}
We have shown that $\omega_* \ll \sigma$, therefore the Radon-Nykodym derivative $k:=d\omega_{*}/d\sigma$ exists. Moreover, \eqref{duality.eq20} implies for any $y\in \Delta$ that
\[k(y):= \lim_{s\to 0} \frac{\omega_{*}(\Delta(y,s))}{\sigma(\Delta(y,s))} \lesssim  \wt N_*(\nabla u)(y). \]
As a consequence
\[\left( \fint_\Delta k^q d\sigma \right)^\frac1q \lesssim \frac{\omega_{*}(\Delta(x,r))}{r^{d-1}} \left( \fint_\Delta |N_*(\nabla u)|^q d\sigma \right)^\frac1q \lesssim   \frac{\omega_{*}(\Delta(x,r))}{r^{d-1}}  r^{-d/q} \|\wt N_*(\nabla u)\|_{L^q(\partial \Omega,\sigma)}. \]
But by using successively \eqref{gamma*togamma}, the solvability of the regularity problem in $L^q$, and \eqref{normf}, we obtain
\[\|\wt N_*(\nabla u)\|_{L^q(\partial \Omega,\sigma)} \lesssim \|\wt N(\nabla u)\|_{L^q(\partial \Omega,\sigma)} \lesssim \|f\|_{\dot W^{1,q}}  \lesssim r^{d/q-1}.\]
The two last computations show that
\[\left( \fint_\Delta k^q d\sigma \right)^\frac1q \lesssim \frac{\omega_*(\Delta)}{r^d} \approx \frac{\omega_*(\Delta)}{\sigma(\Delta)}\]
because $\sigma$ is a $d$-dimensional Ahlfors regular measure. We proved that $k\in RH_q$, as desired, which concludes the theorem.
\end{proof}

\end{document}